\documentclass[11pt,a4paper]{article}

\usepackage[utf8]{inputenc}
\usepackage[T1]{fontenc}
\usepackage{lmodern}
\usepackage[margin=1in]{geometry}
\usepackage{amsmath,amsfonts,amssymb,amsthm}
\usepackage{graphicx}
\usepackage{cite}
\usepackage{xcolor}
\usepackage[hidelinks]{hyperref}

\allowdisplaybreaks[1]
\numberwithin{equation}{section}

\newtheorem{theorem}{Theorem}[section]
\newtheorem{proposition}[theorem]{Proposition}
\newtheorem{lemma}[theorem]{Lemma}
\theoremstyle{remark}
\newtheorem{remark}[theorem]{Remark}
\theoremstyle{plain}

\newcommand{\R}{\mathbb{R}}

\newcommand{\ip}[2]{\left\langle #1,#2\right\rangle}
\newcommand{\norm}[1]{\left\|#1\right\|}
\newcommand{\pos}[1]{\left[#1\right]_{+}}

\newcommand{\bigO}{\mathcal{O}}
\newcommand{\proj}{\operatorname{Proj}}

\allowdisplaybreaks[1]

\title{Accelerated primal-dual dynamical systems for convex optimization with nonlinear inequality constraints\thanks{This work was supported by the National Natural Science Foundation of China (Grant No. 12601606), Sichuan Science and Technology Program (Grant No. 2025ZNSFSC0813), and the Talent Introduction Project of Xihua University (Grant No. Z241102).}}

\author{Xin He\thanks{School of Science, Xihua University, Chengdu, Sichuan, P.R. China (\href{mailto:hexinuser@163.com}{\texttt{hexinuser@163.com}}).}}

\date{}
\hypersetup{
  pdftitle={Accelerated primal-dual dynamical systems for convex optimization with nonlinear inequality constraints},
  pdfauthor={Xin He}
}

\begin{document}
\maketitle

\begin{abstract}
In this paper, we propose accelerated inertial primal-dual dynamics for convex optimization problems with nonlinear convex inequality constraints. The proposed dynamics combine Nesterov accelerated  vanishing damping $\alpha/t$, a projected augmented-Lagrangian multiplier, and tangent extrapolation of the constraint mapping. Without assuming strong convexity, we establish global well-posedness and simultaneous $\mathcal O(t^{-2})$ decay rates for the constraint violation and the objective residual when $\alpha\geq3$ and the extrapolation parameter satisfies a suitable compatibility condition. Under stricter parameter conditions, with $\alpha>3$, we further prove that the  primal-dual trajectory converges to a Karush--Kuhn--Tucker pair, both residuals decay as $o(t^{-2})$. Numerical experiments illustrate the theoretical results.
\end{abstract}


\section{Introduction}

\subsection{Problem formulation and proposed dynamics}
\label{sec:formulation}

In this paper, we study the convex optimization problem
\begin{equation}\label{ques_main}
  \min_{x\in\R^n} f(x)
  \qquad\text{subject to}\qquad g(x)\leq0,
\end{equation}
where $f:\R^n\to\R$ and the components $g_i:\R^n\to\R$,
$i=1,\ldots,m$, of $g=(g_1,\ldots,g_m)^\top$ are convex and
continuously differentiable, and their gradients are locally Lipschitz continuous. Such problems
arise in quadratically constrained optimization, estimation, and
risk-constrained learning
\cite{Boyd2004,LanSiopt,XuLinear2021,ZhuSiopt}.
 
Let \[\mathcal X=\{x\in\R^n:g(x)\leq0\}\] denote the feasible set of
\eqref{ques_main}.
We assume that the optimal solution set $\mathcal X^\star$ is
nonempty and that Slater's condition holds: there exists
$x^{\rm s}\in\R^n$ such that $g_i(x^{\rm s})<0$ for every
$i=1,\ldots,m$. Denote the optimal value by
$f^\star=\min_{x\in\mathcal X}f(x)$ and define the Lagrangian function
\begin{equation}\label{eq:Lag}
  \mathcal L(x,\lambda)
  =f(x)+\ip{\lambda}{g(x)},
  \qquad (x,\lambda)\in\R^n\times\R_+^m,
\end{equation}
where
$
  \R_+^m
  :=
  \{\lambda\in\R^m:\lambda_i\geq0,\ i=1,\ldots,m\}.
$
For the constraint mapping $g$, we denote its Jacobian matrix by
\[
  J_g(x)
  :=
  \bigl[\nabla g_1(x),\ldots,\nabla g_m(x)\bigr]^\top
  \in\R^{m\times n},
\]
so that
\[
  J_g(x)^\top\lambda
  =
  \sum_{i=1}^m\lambda_i\nabla g_i(x).
\]
We next recall the primal-dual characterization of
\eqref{ques_main}. A pair
$(x^\star,\lambda^\star)\in\R^n\times\R_+^m$ is called a
Karush-Kuhn-Tucker (KKT) pair of \eqref{ques_main} if
\begin{equation}\label{eq:KKT}
\left\{
\begin{aligned}
  &0\in \nabla  f(x^\star)+J_g(x^\star)^\top\lambda^\star,\\
  &g(x^\star)\leq0,\quad\lambda^\star\geq0,
  \quad
  \langle\lambda^\star,g(x^\star)\rangle=0.
\end{aligned}
\right.
\end{equation}
Under the standing assumptions, the set $\Omega$ of KKT pairs is
nonempty. Its elements are precisely the saddle points of
$\mathcal L$ \cite{Rockafellar}, that is,
\begin{equation}\label{eq:saddle}
  \mathcal L(x^\star,\lambda)
  \leq\mathcal L(x^\star,\lambda^\star)
  \leq\mathcal L(x,\lambda^\star),
  \qquad x\in\R^n,\quad\lambda\in\R_+^m.
\end{equation}

Continuous-time dynamical systems provide a framework for understanding
how inertial and damping influence the convergence of optimization
methods. Energy estimates along their trajectories can help identify
conditions for accelerated convergence, while suitable time
discretizations can guide the design of numerical algorithms
\cite{SuJmlr,WilsonJmlr,LuoMathProg,ShiMp,WibisonoPnas}. Much of the existing research
on accelerated inertial dynamics has focused on unconstrained
optimization \cite{SuJmlr,AttouchMathProg,MayTjm,AttouchJEMS} and convex optimization
with linear equality constraints \cite{ZengTac,HeSioc,BotJde,LuoJogo,LuoMC,ZhaoJmlr,AttouchJota}.
This motivates further investigation of inertial dynamics for
\eqref{ques_main}, both to understand their convergence properties
under nonlinear inequality constraints and to provide new insights
into the design of accelerated algorithms. We pursue this direction
by developing second-order primal-dual dynamics for the joint
evolution of the primal variable and its multiplier. Our construction starts from the
Powell--Hestenes--Rockafellar (PHR) augmented Lagrangian
\cite{Powell1969,Hestenes1969,RockALM1976},
\[
  \mathcal L_\beta(x,\lambda)
  :=
  f(x)
  +
  \frac{1}{2\beta}
  \left(
    \norm{\pos{\lambda+\beta g(x)}}^2
    -
    \norm{\lambda}^2
  \right),
  \qquad
  (x,\lambda)\in\R^n\times\R^m.
\]
Its primal and dual gradients are
\begin{equation}\label{eq:PHRG}
\begin{aligned}
  \nabla_x\mathcal L_\beta(x,\lambda)
  &=
  \nabla f(x)
  +
  J_g(x)^\top\pos{\lambda+\beta g(x)},\\
  \nabla_\lambda\mathcal L_\beta(x,\lambda)
  &=
  \frac{1}{\beta}
  \left(
    \pos{\lambda+\beta g(x)}
    -
    \lambda
  \right).
\end{aligned}
\end{equation}
The identity
$\lambda=\pos{\lambda+\beta g(x)}$ is equivalent to primal
feasibility, dual nonnegativity, and complementarity. Consequently,
the equilibria of the associated primal-dual gradient dynamics
are precisely the KKT pairs of \eqref{ques_main}. To introduce inertia, we combine vanishing damping with a velocity
extrapolation of the dual variable and a tangent extrapolation of
the nonlinear constraint mapping. Specifically, for $t\geq t_0>0$,
we propose the following accelerated inertial primal-dual dynamic:
\begin{equation}\label{eq:dy}
  \boxed{
  \left\{
  \begin{aligned}
    \ddot x(t)+\frac{\alpha}{t}\dot x(t)
      +\nabla f(x(t))+J_g(x(t))^\top p(t)&=0,\\
    \ddot\lambda(t)+\frac{\alpha}{t}\dot\lambda(t)
      -\frac{\sigma}{\beta}
       \bigl(p(t)-\widehat\lambda(t)\bigr)&=0,
  \end{aligned}
  \right.
  }
\end{equation}
where
\begin{align}
  \widehat\lambda(t)
  &=\lambda(t)+\frac{t}{\gamma}\dot\lambda(t),
  \label{eq:hatL}\\
  \widehat g(t)
  &=g(x(t))+\frac{t}{\gamma}J_g(x(t))\dot x(t),
  \label{eq:hatg}\\
  p(t)
  &=\pos{\widehat\lambda(t)+\beta\widehat g(t)},
  \label{eq:p}
\end{align}
with initial conditions \[(x(t_0),\lambda(t_0), \dot x(t_0), \dot\lambda(t_0))=(x_0,\lambda_0,v_0,w_0)\in \R^n\times\R^m\times\R^n\times\R^m,\] 
and parameters satisfying
\begin{equation}\label{eq:para}
  \beta>0,\qquad
  \sigma>0,\qquad
  \alpha\geq3,\qquad
  2\leq\gamma\leq\alpha-1.
\end{equation}

The vanishing damping $\alpha/t$ follows the continuous-time interpretation
of Nesterov acceleration and has been extensively used in
inertial dynamics; see, e.g.,
\cite{SuJmlr,Bot26Mp,HeSioc,AttouchMathProg,WibisonoPnas}. The
parameter $\beta$ is inherited from the PHR augmented Lagrangian, while
$\sigma>0$ provides a constant scaling of the dual dynamics.
The velocity-extrapolation parameter $t/\gamma$ for the primal and dual variables is a standard
ingredient in accelerated primal-dual dynamics for linearly constrained
problems \cite{ZengTac,BotJde,HeSioc,AttouchJota,ZhaoJmlr,LuoJogo}. Motivated by this
construction, $\widehat\lambda(t)$ extrapolates the dual variable along its
velocity. For the nonlinear constraint mapping, we introduce the tangent
extrapolation $\widehat g(t)$ in \eqref{eq:hatg}, for which
\[
  tJ_g(x(t))\dot x(t)
  =
  \gamma
  \bigl(
    \widehat g(t)-g(x(t))
  \bigr).
\]
This identity provides the compatibility between the nonlinear constraint
coupling and the inertial terms that is needed in the Lyapunov analysis.
Finally, $p(t)\in\R_+^m$ is the projected multiplier associated with the
velocity-corrected PHR structure. In view of \eqref{eq:PHRG},
$
  \frac1\beta
  \bigl(
    p(t)-\widehat\lambda(t)
  \bigr)
$
is the corresponding dual residual, while
$
  \nabla f(x(t))
  +
  J_g(x(t))^\top p(t)
$
defines the primal coupling term. Hence, the dynamic \eqref{eq:dy} combines the
classical Nesterov-type inertial mechanism with the PHR multiplier
structure, together with the tangent extrapolation specifically introduced
to accommodate the nonlinear inequality constraints.

\subsection{Related work}
Nesterov's accelerated gradient method
\cite{NesterovSovMath,NesterovBook} is one of the fundamental developments
in first-order convex optimization. For smooth convex minimization, it
achieves the $\mathcal O(1/k^2)$ convergence rate for the objective
residual. A continuous-time interpretation of Nesterov acceleration was
initiated by Su, Boyd, and Cand\`es \cite{SuJmlr}, who showed that the
accelerated gradient method admits the second-order continuous-time model
\begin{equation}\label{eq:avd}
  \ddot x(t)
  +\frac{\alpha}{t}\dot x(t)
  +\nabla f(x(t))
  =0.
\end{equation}
For the critical value $\alpha=3$, the dynamic \eqref{eq:avd} can be viewed as
the continuous-time counterpart of Nesterov's accelerated gradient method,
while for $\alpha\geq3$ its trajectories satisfy
$
  f(x(t))-\min f=\mathcal O(t^{-2}).
$
For the noncritical regime $\alpha>3$, trajectory convergence can be
combined with strict Lyapunov dissipation to obtain the sharper estimate
$f(x(t))-\min f=o(t^{-2})$; see
\cite{AttouchMathProg,MayTjm}.
This highlights the central role of the vanishing damping $\alpha/t$ in
Nesterov acceleration. Since then, inertial dynamics with vanishing
damping have been extensively studied from different perspectives,
including time scaling \cite{AttouchJEMS,WilsonJmlr,AttouchSicon}, Hessian-driven damping
\cite{BotMp,ShiMp}, and the construction of accelerated algorithms through
numerical discretization \cite{LuoMathProg,ShiMp,HeCOAP}.

 The continuous-time dynamics viewpoint has also been successfully
extended to primal-dual methods for linearly constrained convex
optimization:
\begin{equation}\label{eq:linearP}
  \min_{x} f(x)
  \quad\text{s.t.}\quad
  Ax=b.
\end{equation}
By incorporating Nesterov-type vanishing damping and inertial extrapolation
into the primal-dual framework, accelerated dynamics and their
rate-preserving discretizations have been developed for
\eqref{eq:linearP}; see, e.g.,
\cite{ZengTac,BotJde,HeSioc,HeAuto,BotMathProg,ZhaoJmlr,LuoMC,AttouchJota}. In particular, several of these constructions recover the
Nesterov accelerated convergence rates, yielding $\mathcal O(t^{-2})$ rates for
the objective residual and feasibility violation in continuous time and corresponding
$\mathcal O(k^{-2})$ rates in discrete time. A crucial feature behind this extension is the compatibility between
inertial extrapolation and the affine constraint mapping:
\[
  A\bigl(x(t)+\theta t\dot x(t)\bigr)-b
  =
  Ax(t)-b+\theta tA\dot x(t).
\]
Recently, trajectory convergence and $o(t^{-2})$ objective and feasibility
rates have also been established for Nesterov accelerated primal-dual
dynamics with linear equality constraints in the strict extrapolation regime
 \cite{HeLittleO}. This condition corresponds to
$\alpha>3$.
For a nonlinear mapping $g$, however, extrapolating the primal variable
and then evaluating the constraint mapping is generally not equivalent to
extrapolating the constraint value through its first-order linearization.
Hence, the mechanism that transfers Nesterov acceleration to the
primal-dual framework under affine constraints does not extend directly
to nonlinear inequality constraints.

For problem~\eqref{ques_main} with nonlinear inequality constraints,
much of the primal-dual dynamical literature concerns first-order
flows. Feijer and Paganini \cite{FeijerAutomatica} study the
stability of primal-dual gradient dynamics and their applications
to network optimization. Cherukuri, Mallada, and Cort\'es
\cite{CherukuriScl} establish well-posedness and asymptotic
convergence for projected primal-dual dynamics, while
\cite{CherukuriTac} further investigates the role of convexity
in convergence and robustness. For smooth nonlinear inequalities,
Tang, Qu, and Li \cite{Tang2020} develop augmented primal-dual
gradient dynamics and establish semi-global exponential stability
under strong convexity and additional regularity assumptions.
Beyond these first-order flows, Muehlebach and Jordan
\cite{MuehlebachMathProg} obtain accelerated convergence rates
for nonlinearly constrained convex optimization through a
different dynamical formulation. These results, however, leave
open how to incorporate Nesterov-type inertial and vanishing
damping into an augmented-Lagrangian primal-dual multiplier
framework for problem~\eqref{ques_main}, and how the damping
and extrapolation parameters govern accelerated objective
and feasibility estimates and convergence of the full
primal-dual trajectory under convexity alone.

\subsection{Main contributions}

Existing studies establish Nesterov-type acceleration for
unconstrained and affine-constrained dynamics
\cite{SuJmlr,ZengTac,BotJde,HeSioc}, while augmented
primal-dual gradient dynamics provide a first-order framework
for handling nonlinear inequalities \cite{Tang2020}.
Building on these developments and using the PHR augmented Lagrangian, we propose the second-order
primal-dual system \eqref{eq:dy}, which combines vanishing
damping with tangent extrapolation of the nonlinear constraint
mapping and a projected augmented-Lagrangian multiplier.
The main contributions are as follows.
\begin{enumerate}
\item[(a)]
We propose the Nesterov-type inertial primal-dual system
\eqref{eq:dy}, which combines vanishing damping with tangent
extrapolation of the nonlinear constraint mapping and a
projected augmented-Lagrangian multiplier.
For $\alpha\geq3$ and $2\leq\gamma\leq\alpha-1$, we establish
global well-posedness, boundedness of the primal-dual trajectory,
and the simultaneous estimates
\[
  \|\pos{g(x(t))}\|=\mathcal O(t^{-2}),
  \qquad
  |f(x(t))-f^\star|=\mathcal O(t^{-2}).
\]
The energy analysis exploits the compatibility between
the tangent extrapolation and the projected multiplier
in the inertial coupling. These results extend the accelerated
objective and feasibility estimates for affine-constrained
primal-dual dynamics \cite{ZengTac,BotJde,HeSioc} to nonlinear
convex inequalities, without assuming strong convexity.

\item[(b)]
For the dynamic \eqref{eq:dy}, under the strict parameter
condition $2<\gamma<\alpha-1$, we further prove that the
 primal-dual trajectory converges to a KKT pair and
establish
\[
  \|\pos{g(x(t))}\|+|f(x(t))-f^\star|=o(t^{-2}),
  \qquad
  \|\dot x(t)\|+\|\dot\lambda(t)\|=o(t^{-1}).
\]
Neither the primal solution nor its multiplier is required
to be unique.  We further extend the analysis to additive  perturbations in the primal equation and to  time-dependent scaling of the primal and dual
coupling terms.  To the best of our knowledge, this is the
first Nesterov-type second-order primal-dual multiplier
framework for general convex inequality constraints
that combines simultaneous accelerated objective and feasibility
estimates, their little-$o$ refinements, and convergence of
the  primal-dual trajectory.
\end{enumerate}

 \subsection{Notation and organization}
 For a nonempty closed convex set $C$, let $\proj_C$ denote the Euclidean
projection onto $C$, namely,
$
  \proj_C(z)
  :=
  \arg\min_{y\in C}\|y-z\|.
$
In particular, for $z\in\R^m$, denote
\[
  \pos{z}
  :=
  \proj_{\R_+^m}(z)
  =
  \bigl(\max\{z_1,0\},\ldots,\max\{z_m,0\}\bigr)^\top.
\]
Thus, $\|\pos{g(x)}\|$ provides a natural measure of the nonlinear
inequality violation. Throughout the paper, inequalities between vectors in
$\R^m$ are understood componentwise; that is, for $u,v\in\R^m$,
$u\leq v$ means $u_i\leq v_i$ for all $i=1,\ldots,m$.
A vector-valued function $M:I\to\R^m$, where $I\subset\R$ is an interval,
is said to be componentwise nondecreasing if $M(s)\leq M(t)$ whenever
$s\leq t$. 

 Section~\ref{sec_2} establishes well-posedness and energy estimates. Section~\ref{sec:convergence} proves the large-$O$ rates, trajectory convergence, and little-$o$ improvements. Section~\ref{sec:extensions} extends these results to additive primal perturbations and time-dependent scaling coefficients. Section~\ref{sec:numerics} presents two numerical comparisons, and Section~\ref{sec:conclusion} concludes. Appendix~\ref{sec:appendix} contains the technical lemmas used throughout the proofs.

 \section{Well-posedness and energy estimates}\label{sec_2}

In this section, we establish the well-posedness of the dynamic
\eqref{eq:dy} and derive the energy estimates needed for the
subsequent convergence analysis. 

We first prove local existence
and uniqueness, then construct a Lyapunov function to obtain
a priori bounds and extend the local solution globally.

\begin{proposition}\label{prop:local}
For an initial point $(x_0,\lambda_0,v_0,w_0)$, there exists a unique
solution
\[
  (x,\lambda)\in
  C^2\bigl([t_0,T);\R^n\times\R^m\bigr)
\]
 of the dynamic \eqref{eq:dy} satisfying $(x(t_0),\lambda(t_0))=(x_0,\lambda_0)$ and $(\dot x(t_0), \dot\lambda(t_0))=(v_0,w_0)$ on a maximal
interval $[t_0,T)$, where $T\in(t_0,+\infty]$.
\end{proposition}
\begin{proof}
Set $v=\dot x$ and $w=\dot\lambda$. Then \eqref{eq:dy} is equivalent to the first-order system
\begin{equation}\label{eq:first}
\left\{
\begin{aligned}
  \dot x&=v,\\
  \dot\lambda&=w,\\
  \dot v&=-\frac{\alpha}{t}v-\nabla f(x)
  -J_g(x)^\top
  \pos{\lambda+\frac{t}{\gamma}w+\beta g(x)
  +\frac{\beta t}{\gamma}J_g(x)v},\\
  \dot w&=-\frac{\alpha}{t}w+\frac{\sigma}{\beta}
  \left[
  \pos{\lambda+\frac{t}{\gamma}w+\beta g(x)
  +\frac{\beta t}{\gamma}J_g(x)v}
  -\lambda-\frac{t}{\gamma}w
  \right].
\end{aligned}
\right.
\end{equation}
Write $z=(x,\lambda,v,w)$ and denote the right-hand side
of \eqref{eq:first} by $F(t,z)$.

By assumption, $\nabla f$ and $\nabla g_i$,
$i=1,\ldots,m$, are locally Lipschitz continuous.
Thus $J_g$ is locally Lipschitz continuous, and the continuous
differentiability of $g$ implies that $g$ is also locally
Lipschitz continuous.
Moreover, for any $(x,v)$ and $(\widetilde x,\widetilde v)$,
\[
  \|J_g(x)v-J_g(\widetilde x)\widetilde v\|
  \leq
  \|J_g(x)\|\,\|v-\widetilde v\|+
  \|J_g(x)-J_g(\widetilde x)\|\,\|\widetilde v\|.
\]
Local boundedness and local Lipschitz continuity of $J_g$
therefore imply local Lipschitz continuity of
$(x,v)\mapsto J_g(x)v$.

Since projection onto the nonnegative orthant is nonexpansive,
\[
  \|\pos{a}-\pos{b}\|\leq\|a-b\|,
  \qquad a,b\in\R^m,
\]
the projected multiplier in \eqref{eq:first} is locally
Lipschitz continuous in $z$, uniformly for $t$ in compact
subintervals of $[t_0,+\infty)$.
The same property holds for its product with $J_g(x)^\top$,
because both factors are locally bounded and locally
Lipschitz continuous.
Consequently, $F$ is continuous in $(t,z)$ and locally
Lipschitz continuous in $z$, uniformly on compact time
intervals; here $t_0>0$ ensures that the damping coefficient
$\alpha/t$ is nonsingular. The Picard--Lindel\"of theorem and the standard continuation
theorem yield a unique maximal solution
$z\in C^1([t_0,T);\R^{2n+2m})$.
Since $\dot x=v$ and $\dot\lambda=w$, it follows that
\[(x,\lambda)\in C^2([t_0,T);\R^n\times\R^m).\]
\end{proof}

We next derive an energy estimate for the maximal solution. For an initial point $(x_0,\lambda_0,v_0,w_0)$, let
$(x(t),\lambda(t))$ be the maximal solution of the dynamic \eqref{eq:dy} given by
Proposition~\ref{prop:local}, defined on $[t_0,T)$, and let
$(x^\star,\lambda^\star)\in \Omega$. Write
$z^\star=(x^\star,\lambda^\star)$ to make the dependence of the
Lyapunov function on the reference KKT pair explicit, and define
\begin{equation}\label{eq:E}
\begin{aligned}
  {\mathcal E_{z^\star}(t)}:={}&
  t^2\bigl(
    \mathcal L(x(t),\lambda^\star)
    -\mathcal L(x^\star,\lambda^\star)
  \bigr)+\frac12
  \|\gamma(x(t)-x^\star)+t\dot x(t)\|^2
  \\
  &+\frac{\gamma\delta}{2}\|x(t)-x^\star\|^2+\frac1{2\sigma}
  \|\gamma(\lambda(t)-\lambda^\star)
  +t\dot\lambda(t)\|^2
  +\frac{\gamma\delta}{2\sigma}
  \|\lambda(t)-\lambda^\star\|^2.
\end{aligned}
\end{equation}
Here $\mathcal L$ denotes the Lagrangian associated with
\eqref{ques_main}, defined in \eqref{eq:Lag}, and
\begin{equation}\label{eq_delta}
	\delta=\alpha-\gamma-1\geq0.
\end{equation}
It follows from \eqref{eq:saddle} and \eqref{eq:para} that
${\mathcal E_{z^\star}(t)}\geq0$ for every $t\in[t_0,T)$. The following result establishes global well-posedness of
\eqref{eq:dy} and the basic estimates used in the subsequent
convergence analysis.

\begin{theorem}\label{thm:global}
Let $(x,\lambda)\in
C^2([t_0,T);\R^n\times\R^m)$ be the maximal solution of
\eqref{eq:dy} and $(x^\star,\lambda^\star)\in \Omega$. Then the following statements hold:
\begin{itemize}
  \item[(i)]
  The solution is global, namely, $T=+\infty$.
  \item[(ii)]
  The energy function ${\mathcal E_{z^\star}}$ is nonincreasing on $[t_0,+\infty)$.

  \item[(iii)] $(x(t),\lambda(t))$ is bounded on
  $[t_0,+\infty)$ and
\[
    \|\dot x(t)\|+\|\dot\lambda(t)\|
    =
    \mathcal O\left(\frac1t\right).
\]
\end{itemize}
\end{theorem}

\begin{proof}
We first establish the Lyapunov estimate on the maximal interval
$[t_0,T)$. From the primal equation in \eqref{eq:dy} and \eqref{eq_delta},
\begin{align}
  \frac{d}{dt}
  \bigl[
    \gamma(x(t)-x^\star)+t\dot x(t)
  \bigr]
  &=
  (\gamma+1)\dot x(t)+t\ddot x(t)\notag\\
  &=
  -\delta\dot x(t)
  -t\bigl(
    \nabla f(x(t))+J_g(x(t))^\top p(t)
  \bigr).
  \label{eq:Ix}
\end{align}
Similarly, the dual equation yields
\begin{align}  \label{eq:Il}
  \frac{d}{dt}
  \bigl[
    \gamma(\lambda(t)-\lambda^\star)
    +t\dot\lambda(t)
  \bigr]
  =
  -\delta\dot\lambda(t)
  +\frac{\sigma t}{\beta}
  \bigl(
    p(t)-\widehat\lambda(t)
  \bigr).
\end{align}
Moreover, by \eqref{eq:Lag} and $\frac{d}{dt}g(x(t))
  =
  J_g(x(t))\dot x(t)$, we have
\begin{align}
  &\frac{d}{dt}
  \left[
    t^2
    \bigl(
      \mathcal L(x(t),\lambda^\star)
      -\mathcal L(x^\star,\lambda^\star)
    \bigr)
  \right]\notag\\
  &\qquad=
  2t
  \bigl(
    \mathcal L(x(t),\lambda^\star)
    -\mathcal L(x^\star,\lambda^\star)
  \bigr)
  +
  t^2
  \ip{
    \nabla f(x(t))
    +J_g(x(t))^\top\lambda^\star
  }{
    \dot x(t)
  }.
  \label{eq:Ldot}
\end{align}

Differentiating \eqref{eq:E} and using
\eqref{eq:Ix}-\eqref{eq:Ldot}, we obtain
\begin{align}\label{eq:Eraw}
  \dot{{\mathcal E}}_{z^\star}(t)={}&
  2t
  \bigl(
    \mathcal L(x(t),\lambda^\star)
    -\mathcal L(x^\star,\lambda^\star)
  \bigr)+
  t^2
  \ip{
    \nabla f(x(t))
    +J_g(x(t))^\top\lambda^\star
  }{
    \dot x(t)
  }\notag\\
  &+
  \ip{
    \gamma(x(t)-x^\star)+t\dot x(t)
  }{
    -\delta\dot x(t)
    -t\bigl(
      \nabla f(x(t))
      +J_g(x(t))^\top p(t)
    \bigr)
  }\notag\\
  &+
  \frac1\sigma
  \ip{
    \gamma(\lambda(t)-\lambda^\star)
    +t\dot\lambda(t)
  }{
    -\delta\dot\lambda(t)
    +\frac{\sigma t}{\beta}
    \bigl(
      p(t)-\widehat\lambda(t)
    \bigr)
  }\\
  &+
  \gamma\delta
  \ip{x(t)-x^\star}{\dot x(t)}
  +
  \frac{\gamma\delta}{\sigma}
  \ip{\lambda(t)-\lambda^\star}{\dot\lambda(t)}.\notag
\end{align}
The terms involving $\gamma\delta$ cancel. By
\eqref{eq:hatL},
$
  \gamma(\lambda(t)-\lambda^\star)
  +t\dot\lambda(t)
  =
  \gamma(\widehat\lambda(t)-\lambda^\star)
$
and
\begin{align*}
  &t^2
  \ip{
    \nabla f(x(t))
    +J_g(x(t))^\top\lambda^\star
  }{
    \dot x(t)
  }
  -
  t^2
  \ip{
    \dot x(t)
  }{
    \nabla f(x(t))
    +J_g(x(t))^\top p(t)
  }\\
  &\qquad=
  t^2
  \ip{
    J_g(x(t))^\top(\lambda^\star-p(t))
  }{
    \dot x(t)
  }.
\end{align*}
Hence, from \eqref{eq:Eraw},
\begin{align}  \label{eq:Eexp}
  \dot{{\mathcal E}}_{z^\star}(t)={}&
  2t
  \bigl(
    \mathcal L(x(t),\lambda^\star)
    -\mathcal L(x^\star,\lambda^\star)
  \bigr)
  -\gamma t
  \ip{x(t)-x^\star}{\nabla f(x(t))}\notag\\
  &-\gamma t
  \ip{
    x(t)-x^\star
  }{
    J_g(x(t))^\top p(t)
  }
  +t^2
  \ip{
    J_g(x(t))^\top(\lambda^\star-p(t))
  }{
    \dot x(t)
  }\\
  &+
  \frac{\gamma t}{\beta}
  \ip{
    \widehat\lambda(t)-\lambda^\star
  }{
    p(t)-\widehat\lambda(t)
  }
  -\delta t
  \left(
    \|\dot x(t)\|^2
    +\frac1\sigma\|\dot\lambda(t)\|^2
  \right).\notag
\end{align}

We next estimate the objective and constraint terms. By the convexity of
$f$,
\begin{equation}\label{eq:fconv}
  \ip{\nabla f(x(t))}{x(t)-x^\star}
  \geq
  f(x(t))-f(x^\star).
\end{equation}
For each $i=1,\ldots,m$, the convexity of $g_i$ gives
\[
  g_i(x^\star)
  \geq
  g_i(x(t))
  +
  \ip{
    \nabla g_i(x(t))
  }{
    x^\star-x(t)
  }.
\]
Since $p(t)\in\R_+^m$ and $g(x^\star)\leq0$, multiplying by
$p_i(t)$ and summing over $i$ yields
\begin{equation}\label{eq:conv}
  -\ip{
    x(t)-x^\star
  }{
    J_g(x(t))^\top p(t)
  }
  \leq
  \ip{
    p(t)
  }{
    g(x^\star)-g(x(t))
  }
  \leq
  -\ip{p(t)}{g(x(t))}.
\end{equation}
Furthermore, by \eqref{eq:hatg},
\begin{align}
  t^2
  \ip{
    J_g(x(t))^\top(\lambda^\star-p(t))
  }{
    \dot x(t)
  }
  &=
  t^2
  \ip{
    \lambda^\star-p(t)
  }{
    J_g(x(t))\dot x(t)
  }\notag\\
  &=
  \gamma t
  \ip{
    \lambda^\star-p(t)
  }{
    \widehat g(t)-g(x(t))
  }.
  \label{eq:tangent}
\end{align}
Since the KKT conditions \eqref{eq:KKT} imply
$\ip{\lambda^\star}{g(x^\star)}=0$, we also have
\[
  \mathcal L(x(t),\lambda^\star)
  -
  \mathcal L(x^\star,\lambda^\star)
  =
  f(x(t))-f(x^\star)
  +
  \ip{\lambda^\star}{g(x(t))}.
\]
Substituting \eqref{eq:fconv}-\eqref{eq:tangent} into
\eqref{eq:Eexp} and collecting terms, we obtain
\begin{align}
  \dot{{\mathcal E}}_{z^\star}(t)\leq{}&
  -(\gamma-2)t
  \bigl(
    \mathcal L(x(t),\lambda^\star)
    -\mathcal L(x^\star,\lambda^\star)
  \bigr)
  -\delta t
  \left(
    \|\dot x(t)\|^2
    +\frac1\sigma\|\dot\lambda(t)\|^2
  \right)\notag \\
  &+
  \gamma t
  \left[
    \ip{
      \widehat g(t)
    }{
      \lambda^\star-p(t)
    }
    +
    \frac1\beta
    \ip{
      \widehat\lambda(t)-\lambda^\star
    }{
      p(t)-\widehat\lambda(t)
    }
  \right]
    \label{eq:Eproj}
\end{align}
Since
\[
  p(t)
  =\pos{
    \widehat\lambda(t)
    +
    \beta\widehat g(t)
  }=
  \operatorname{proj}_{\R_+^m}
  \bigl(
    \widehat\lambda(t)+\beta\widehat g(t)
  \bigr)
\]
and $\lambda^\star\in\R_+^m$, Lemma~\ref{lem:proj} gives
\begin{equation}\label{eq:projchar}
  \ip{
    \widehat\lambda(t)
    +\beta\widehat g(t)
    -p(t)
  }{
    \lambda^\star-p(t)
  }
  \leq0.
\end{equation}
Moreover,
\[
  \ip{
    \widehat\lambda(t)-\lambda^\star
  }{
    p(t)-\widehat\lambda(t)
  }
  =
  \ip{
    \lambda^\star-p(t)
  }{
    \widehat\lambda(t)-p(t)
  }
  -
  \|p(t)-\widehat\lambda(t)\|^2.
\]
Hence, from \eqref{eq:projchar},
\begin{align} \label{eq:projest}
  &\ip{
    \widehat g(t)
  }{
    \lambda^\star-p(t)
  }
  +
  \frac1\beta
  \ip{
    \widehat\lambda(t)-\lambda^\star
  }{
    p(t)-\widehat\lambda(t)
  }
  \notag\\
  &\qquad=
  \frac1\beta
  \ip{
    \widehat\lambda(t)
    +\beta\widehat g(t)
    -p(t)
  }{
    \lambda^\star-p(t)
  }
  -
  \frac1\beta
  \|p(t)-\widehat\lambda(t)\|^2
  \\
  &\qquad \leq
  -\frac1\beta
  \|p(t)-\widehat\lambda(t)\|^2.\notag
\end{align}
Combining \eqref{eq:Eproj} and \eqref{eq:projest}, we obtain, for every $t\in[t_0,T)$,
\begin{align}\label{eq:Edot}
  \dot{{\mathcal E}}_{z^\star}(t)&\leq{}
  -(\gamma-2)t
  \bigl(
    \mathcal L(x(t),\lambda^\star)
    -\mathcal L(x^\star,\lambda^\star)
  \bigr)\nonumber \\
 & \quad-\delta t
  \left(
    \|\dot x(t)\|^2
    +\frac1\sigma\|\dot\lambda(t)\|^2
  \right)
  -\frac{\gamma t}{\beta}
  \|p(t)-\widehat\lambda(t)\|^2\\
  &\leq0.\nonumber
\end{align}
Therefore, ${\mathcal E_{z^\star}}$ is nonincreasing on $[t_0,T)$, and
\[
  0\leq
  {\mathcal E_{z^\star}(t)}
  \leq
  {\mathcal E_{z^\star}(t_0)}
\] for every $t\in[t_0,T)$.

In particular, from \eqref{eq:E} and $\delta\ge0$,
\[
  \|\gamma(x(t)-x^\star)+t\dot x(t)\|
  \leq
  \sqrt{2{\mathcal E_{z^\star}(t_0)}},\quad
  \|\gamma(\lambda(t)-\lambda^\star)+t\dot\lambda(t)\|
  \leq
  \sqrt{2\sigma{\mathcal E_{z^\star}(t_0)}}.
\]
Applying Lemma~\ref{le_bounde} with $a(t)=t/\gamma$
to the primal and dual estimates above shows that
$x(t)$ and $\lambda(t)$ are bounded on $[t_0,T)$ and that
\begin{equation}\label{eq:tvel}
  \sup_{t\in[t_0,T)}
  t\|\dot x(t)\|
  <+\infty,
  \qquad
  \sup_{t\in[t_0,T)}
  t\|\dot\lambda(t)\|
  <+\infty.
\end{equation}
Moreover, since $g$ and $J_g$ are
continuous, \eqref{eq:dy}-\eqref{eq:para} imply that
$\widehat\lambda(t)$ and $\widehat g(t)$
are also bounded. Consequently,
$p(t)$ is bounded on $[t_0,T)$.

We now prove that the maximal solution is global. Suppose, by
contradiction, that $T<+\infty$. Since $x$, $\dot x$, $\lambda$,
$\dot\lambda$, $\hat\lambda$, and $p$ are bounded on $[t_0,T)$, and since
$\nabla f$ and $J_g$ are continuous, \eqref{eq:dy} implies that
$\ddot x$ and $\ddot\lambda$ are bounded on $[t_0,T)$. Consequently,
$x$, $\dot x$, $\lambda$, and $\dot\lambda$ admit finite limits as
$t\uparrow T$. Define
\[
  \bigl(
    x(T),\dot x(T),\lambda(T),\dot\lambda(T)
  \bigr)
  :=
  \lim_{t\uparrow T}
  \bigl(
    x(t),\dot x(t),\lambda(t),\dot\lambda(t)
  \bigr).
\]
Since $T>0$, Proposition~\ref{prop:local} guarantees a local solution
starting from these values at time $T$. This extends the original solution
beyond $T$, contradicting the maximality of $[t_0,T)$. Hence
$T=+\infty$, which proves~(i).

Consequently, \eqref{eq:Edot} holds for every $t\geq t_0$, and
${\mathcal E_{z^\star}}$ is nonincreasing on $[t_0,+\infty)$, proving~(ii).
The estimates  therefore hold on the entire interval
$[t_0,+\infty)$. In particular, $(x(t),\lambda(t))$ is bounded and, by
\eqref{eq:tvel},
\[
  \|\dot x(t)\|+\|\dot\lambda(t)\|
  =
  \mathcal O(t^{-1}).
\]
\end{proof}
\begin{remark}
Theorem~\ref{thm:global} ensures that the dynamic \eqref{eq:dy}
is well defined for all $t\geq t_0$ and that both the
primal and dual trajectories remain bounded. The monotonicity
of $\mathcal E_{z^\star}$ and the $O(t^{-1})$ velocity estimate
provide the basic bounds for the subsequent convergence
analysis. 
\end{remark}

\section{Convergence analysis}\label{sec:convergence}
Building on the energy estimates established in Section~\ref{sec_2}, we first derive simultaneous $\bigO(t^{-2})$ bounds
for the nonlinear feasibility violation and the 
objective residual. Under stricter conditions on the damping
and extrapolation parameters, we then establish convergence
of the primal-dual trajectory to a KKT pair and
sharpen the residual and velocity estimates to $o(t^{-2})$
and $o(t^{-1})$, respectively.

\subsection{Convergence rates for nonlinear feasibility violation and objective residual}

To establish convergence rates for \eqref{ques_main}, we first derive an
$\mathcal O(t^{-2})$ bound for the nonlinear feasibility violation
$\|\pos{g(x(t))}\|$ from the dual projection structure. Combining this
estimate with the saddle-point property \eqref{eq:saddle} then yields the
same rate for $|f(x(t))-f^\star|$.

\begin{theorem}
\label{thm:rates}
Let $(x(t),\lambda(t))$ be the unique global solution of the dynamic
\eqref{eq:dy} and $(x^\star,\lambda^\star)\in\Omega$. Then, as
$t\to+\infty$,
\[
  \|\pos{g(x(t))}\|
  =
  \mathcal O\left(\frac{1}{t^{2}}\right),
  \qquad
  |f(x(t))-f^\star|
  =
  \mathcal O\left(\frac{1}{t^{2}}\right).
\]
\end{theorem}

\begin{proof}
We first establish the constraint violation estimate. Using 
\[\pos{y}-y=\pos{-y}\] and \eqref{eq:p}, define
\begin{equation}\label{eq:eta}
  \eta(t)
  :=
  \frac1\beta
  \bigl(
    p(t)-\widehat\lambda(t)-\beta\widehat g(t)
  \bigr)
  =
  \frac1\beta
  \pos{
    -\widehat\lambda(t)-\beta\widehat g(t)
  }
  \in\R_+^m.
\end{equation}
Differentiating \eqref{eq:hatL} and using the dual equation in
\eqref{eq:dy}, we obtain
\begin{equation}\label{eq:dotHL}
\begin{aligned}
	 \gamma\dot{\widehat\lambda}(t)
  &=
  (\gamma+1)\dot\lambda(t)+t\ddot\lambda(t)\\
 & =
  -\delta\dot\lambda(t)
  +
  \frac{\sigma t}{\beta}
  \bigl(
    p(t)-\widehat\lambda(t)
  \bigr)\\
  &= -\delta\dot\lambda(t)
  +
  \sigma t\widehat g(t)
  +
  \sigma t\eta(t).
\end{aligned}
\end{equation}
By \eqref{eq:hatg},
$
  \gamma t\widehat g(t)
  =
  \gamma t g(x(t))
  +
  t^2\frac{d}{dt}g(x(t)).
$
Consequently, \eqref{eq:dotHL} gives
\begin{equation}\label{eq:dtg}
  t^2\frac{d}{dt}g(x(t))
  +
  \gamma t g(x(t))
  =
  \frac{\gamma^2}{\sigma}
  \dot{\widehat\lambda}(t)
  +
  \frac{\gamma\delta}{\sigma}
  \dot\lambda(t)
  -
  \gamma t\eta(t).
\end{equation}

Set
$G(t):=t^2g(x(t)).
$
Since
$
  \dot G(t)
  =
  2tg(x(t))
  +
  t^2\frac{d}{dt}g(x(t)),
$
\eqref{eq:dtg} becomes
\[
  \dot G(t)
  +
  \frac{\gamma-2}{t}G(t)
  =
  \frac{\gamma^2}{\sigma}
  \dot{\widehat\lambda}(t)
  +
  \frac{\gamma\delta}{\sigma}
  \dot\lambda(t)
  -
  \gamma t\eta(t).
\]
Integrating over $[t_0,t]$ yields
\begin{equation}\label{eq:Volterra}
  G(t)
  +
  \int_{t_0}^t
  \frac{\gamma-2}{s}G(s)\,ds
  =
  B(t)-M(t),
\end{equation}
where 
\begin{equation}\label{eq:defMB}
\begin{aligned}
	M(t)&=
  \gamma
  \int_{t_0}^t
  s\eta(s)\,ds,\\
    B(t)&=
  t_0^2g(x(t_0))
  +
  \frac{\gamma^2}{\sigma}
  \bigl(
    \widehat\lambda(t)-\widehat\lambda(t_0)
  \bigr)
  +
  \frac{\gamma\delta}{\sigma}
  \bigl(
    \lambda(t)-\lambda(t_0)
  \bigr).
\end{aligned}
\end{equation}
By \eqref{eq:hatL} and Theorem~\ref{thm:global}, $\lambda(t)$ and $\widehat\lambda(t)$ are bounded,
and hence $B(t)$ is bounded. Moreover, \eqref{eq:eta} gives $\eta(t)\in\R_+^m$, so $M(t)$ is
componentwise nondecreasing, while $(\gamma-2)/t\geq0$.
Lemma~\ref{lem:volterra}(i), applied to \eqref{eq:Volterra}, therefore gives
\[
  \sup_{t\geq t_0}
  \|\pos{G(t)}\|
  \leq
  2
  \sup_{t\geq t_0}
  \|B(t)\|
  <+\infty.
\]
It follows that
\begin{equation}\label{eq:feasrate}
  \|\pos{g(x(t))}\|
  = \frac{1}{t^2}\|\pos{G(t)}\|=
  \mathcal O(t^{-2}).
\end{equation}

We next estimate the objective residual. Since the trajectory need not be
feasible, $f(x(t))-f^\star$ is not necessarily nonnegative, and its upper
and lower bounds are considered separately.

For the upper estimate,  we specialize the auxiliary energy
$\mathcal E_{z^\star}$ by taking
$
\lambda^\star=0\in\mathbb R^m_+,
$
that is, \(z^\star=(x^\star,0)\), and define
\begin{equation}\label{eq:E0}
  \mathcal E_0(t)
  :=
  \mathcal E_{z^\star}(t)\big|_{\lambda^\star=0}
  =
  t^2
  \bigl(
    f(x(t))-f^\star
  \bigr)
  +
  Q(t),
\end{equation}
with
\[
  Q(t)={}
  \frac12
  \|\gamma(x(t)-x^\star)+t\dot x(t)\|^2
  +
  \frac{\gamma\delta}{2}
  \|x(t)-x^\star\|^2+
  \frac1{2\sigma}
  \|\gamma\lambda(t)+t\dot\lambda(t)\|^2
  +
  \frac{\gamma\delta}{2\sigma}
  \|\lambda(t)\|^2.
\]
Note that the function $\mathcal E_0$ has the same structure as
${\mathcal E_{z^\star}}$ in \eqref{eq:E}, with the reference multiplier
$\lambda^\star$ replaced by $0$. Although $(x^\star,0)$ need not be a KKT pair, 
the calculation leading to \eqref{eq:Edot} can be
repeated with the reference multiplier equal to $0$,
since $x^\star$ is feasible and
$\ip{0}{g(x^\star)}=0$. The projection estimate remains valid
and yields
\[
  \dot{\mathcal E}_0(t)
  \leq
  -(\gamma-2)t\bigl(f(x(t))-f^\star\bigr).
\]
By Theorem~\ref{thm:global}, there exists $K>0$ such that
$
  0\leq
  Q(t)
  \leq
  K$ for every $t\geq t_0$. Using \eqref{eq:E0}, we obtain
\[
  \dot{\mathcal E}_0(t)
  \leq
  -\frac{\gamma-2}{t}
  \mathcal E_0(t)
  +
  \frac{(\gamma-2)K}{t}.
\]
Multiplying both sides by $t^{\gamma-2}$ gives
\[
  \frac{d}{dt}
  \bigl(
    t^{\gamma-2}\mathcal E_0(t)
  \bigr)
  \leq
  (\gamma-2)Kt^{\gamma-3}.
\]
Since $\gamma\geq2$, integration over $[t_0,t]$ yields
\[
  \mathcal E_0(t)
  \leq
  K
  +
  \left(
    \frac{t_0}{t}
  \right)^{\gamma-2}
  \bigl(
    \mathcal E_0(t_0)-K
  \bigr)
  \leq
  \max\{K,\mathcal E_0(t_0)\}.
\]
Since $Q(t)\geq0$, \eqref{eq:E0} gives
\begin{equation}\label{eq:upper}
    f(x(t))-f^\star
  =
  \frac{\mathcal E_0(t)-Q(t)}{t^2}
  \leq
  \frac{\mathcal E_0(t)}{t^2}
  \leq
  \frac{\max\{K,\mathcal E_0(t_0)\}}{t^2}.
\end{equation}

For the lower estimate, the saddle-point property \eqref{eq:saddle} and $\lambda^\star\in\R_+^m$ yield
\[
  f(x(t))-f^\star
  \geq
  -\ip{\lambda^\star}{g(x(t))}
  \geq
  -\ip{\lambda^\star}{\pos{g(x(t))}}.
\]
Hence, by \eqref{eq:feasrate}, there exists $C>0$ such that
\begin{equation}\label{eq:lower}
  f(x(t))-f^\star
  \geq
  -\|\lambda^\star\|
  \|\pos{g(x(t))}\|
  \geq
  -\frac{C}{t^2}.
\end{equation}
Combining \eqref{eq:upper} and \eqref{eq:lower} yields
$
  |f(x(t))-f^\star|
  =
  \mathcal O(t^{-2}).
$
\end{proof}

\begin{remark}
The proposed dynamic \eqref{eq:dy} and
Theorem~\ref{thm:rates} extend Nesterov-type inertial approaches for
unconstrained and linearly constrained convex optimization
\cite{SuJmlr,ZengTac,BotJde,HeSioc} to problem \eqref{ques_main}. They are
also distinct from the velocity-constrained dynamics of
\cite{MuehlebachMathProg}: here the inequality structure is represented by
an evolving multiplier and a projected PHR residual. For a genuinely nonlinear $g$, identity
\eqref{eq:hatg} supplies the replacement needed in the Lyapunov estimate.
The theorem therefore gives simultaneous $\mathcal O(t^{-2})$ rates for
the objective residual and nonlinear feasibility without strong
convexity.
\end{remark}

\subsection{Trajectory convergence and improved rates}

Throughout this subsection, we assume that
\[
  \alpha>3,
  \qquad
  2<\gamma<\alpha-1.
\]
Under these conditions, we prove convergence of the  primal-dual
trajectory and improve the objective and feasibility estimates from
$\mathcal O(t^{-2})$ to $o(t^{-2})$.

Let $(x(t),\lambda(t))$ be the unique global solution of
\eqref{eq:dy}, and write $z(t)=(x(t),\lambda(t))$.
For each $z^\star=(x^\star,\lambda^\star)\in\Omega$, set
\begin{equation}\label{eq_defd}
\begin{aligned}
  \Phi_{z^\star}(t)
  &:=
  \mathcal L(x(t),\lambda^\star)
  -\mathcal L(x^\star,\lambda^\star)\geq0,\\
  d(t)&:=p(t)-\widehat\lambda(t).
\end{aligned}
\end{equation}
The energy $\mathcal E_{z^\star}$ in \eqref{eq:E} is used throughout
this part with the same reference pair $z^\star$. From \eqref{eq:Edot},
\begin{equation}\label{eq:strict_diss}
  \dot{\mathcal E}_{z^\star}(t)
  \leq{}
  -(\gamma-2)t\Phi_{z^\star}(t)
  -\delta t
  \left(
    \|\dot x(t)\|^2
    +\frac1\sigma\|\dot\lambda(t)\|^2
  \right)-\frac{\gamma t}{\beta}\|d(t)\|^2,
  \qquad t\geq t_0.
\end{equation}
Since
$\gamma-2>0$ and $\delta=\alpha-\gamma-1>0$, integration of
\eqref{eq:strict_diss}   over
$[t_0,T]$ and letting $T\to+\infty$, the boundedness of ${\mathcal E}_{z^\star}(t)$ yields 
\begin{equation}\label{eq:strict_integrability}
\begin{aligned}
	& \int_{t_0}^{+\infty}t\Phi_{z^\star}(t)\,dt<+\infty,
  \\
  &\int_{t_0}^{+\infty}t\|\dot z(t)\|^2\,dt<+\infty,
  \\
 & \int_{t_0}^{+\infty}t\|d(t)\|^2\,dt<+\infty.
\end{aligned}
\end{equation}
Moreover, $\mathcal E_{z^\star}$ is nonnegative and nonincreasing.
Consequently,
\begin{equation}\label{eq:energy_limit}
  \lim_{t\to+\infty}\mathcal E_{z^\star}(t)
  \quad\text{exists and is finite}.
\end{equation}
These properties hold for every $z^\star\in\Omega$.

We first identify the cluster points of the trajectory.

\begin{lemma}\label{lem:cluster}
Let $(x(t),\lambda(t))$ be the unique global solution of the dynamic
\eqref{eq:dy}. Then, every cluster point of $(x(t),\lambda(t))$ 
belongs to $\Omega$.
\end{lemma}

\begin{proof}
By Theorem~\ref{thm:global}, $z(t)$ is bounded.
Let $\bar z=(\bar x,\bar\lambda)$ be a cluster point, and choose
$t_k\to+\infty$ such that
\[
  z(t_k)\to\bar z.
\]
Fix $\rho>1$ and set $I_k=[t_k,\rho t_k]$.
By \eqref{eq:strict_integrability} and the continuity of $d(t)$,
Lemma~\ref{lem:A5}
applied with $u(t)=z(t)$ and $r(t)=d(t)$ gives
\begin{equation}\label{eq:uniform_window}
  \sup_{s\in I_k}\|z(s)-\bar z\|
  \leq
  \|z(t_k)-\bar z\|
  +\sup_{s\in I_k}\|z(s)-z(t_k)\|\to0, \quad \text{as } t_k\to+\infty 
\end{equation}
and there exists $s_k\in I_k$ such that
\[
   \lim_{s_k\to+\infty}   s_k\|\dot z(s_k)\|+s_k\|d(s_k)\| = 0.\]
 Since $s_k\in I_k$, \eqref{eq:uniform_window} gives
\[
  x(s_k)\to\bar x,\qquad
  \lambda(s_k)\to\bar\lambda.
\]
Moreover, the choice of $s_k$ implies
\[
  s_k\dot x(s_k)\to0,\qquad
  s_k\dot\lambda(s_k)\to0,\qquad
  d(s_k)=p(s_k)-\widehat\lambda(s_k)\to0,
\]
where the last limit follows from
$s_k\|d(s_k)\|\to0$ and $s_k\to+\infty$.
Thus, by \eqref{eq:hatL},
\[
  \widehat\lambda(s_k)
  =
  \lambda(s_k)+\frac1\gamma s_k\dot\lambda(s_k)
  \to \bar\lambda.
\]
The continuity of $g$ and $J_g$ gives
$g(x(s_k))\to g(\bar x)$ and boundedness of
$\{J_g(x(s_k))\}$. Hence, by \eqref{eq:hatg},
\[
  \|\widehat g(s_k)-g(\bar x)\|
  \leq
  \|g(x(s_k))-g(\bar x)\|
  +\frac1\gamma
   \|J_g(x(s_k))\|\,\|s_k\dot x(s_k)\|
  \to0.
\]
Finally, since $d=p-\widehat\lambda$,
\[
  p(s_k)=\widehat\lambda(s_k)+d(s_k)
  \to \bar\lambda.
\]
Consequently,
\[
  \widehat\lambda(s_k)\to\bar\lambda,
  \qquad
  \widehat g(s_k)\to g(\bar x),
  \qquad
  p(s_k)\to\bar\lambda.
\]
Passing to the limit in the projection identity \eqref{eq:p} gives
\begin{equation}\label{eq:cluster_projection}
  \bar\lambda
  =\pos{\bar\lambda+\beta g(\bar x)}.
\end{equation}
For completeness, \eqref{eq:cluster_projection} is equivalent,
componentwise, to
\[
  g(\bar x)\leq0,\qquad
  \bar\lambda\geq0,\qquad
  \ip{\bar\lambda}{g(\bar x)}=0.
\]
Indeed, if $\bar\lambda_i>0$, then
\eqref{eq:cluster_projection} gives $g_i(\bar x)=0$; if
$\bar\lambda_i=0$, it gives $g_i(\bar x)\leq0$.

It remains to verify stationarity. Set
$\ell_k=(\rho-1)t_k$, the length of $I_k=[t_k,\rho t_k]$.
We first show that
\begin{equation}\label{eq:average_p}
 \lim_{t_k\to+\infty}  \frac1{\ell_k}\int_{I_k}\|p(s)-\bar\lambda\|\,ds
  =0.
\end{equation}
It follows from \eqref{eq:hatL} and \eqref{eq_defd} that
$p(s)=\lambda(s)+(s/\gamma)\dot\lambda(s)+d(s)$, and hence
\begin{align*}
  \frac1{\ell_k}\int_{I_k}\|p(s)-\bar\lambda\|\,ds
  \leq
  \sup_{s\in I_k}\|\lambda(s)-\bar\lambda\|+\frac1{\gamma\ell_k}
    \int_{I_k}s\|\dot\lambda(s)\|\,ds
  +\frac1{\ell_k}\int_{I_k}\|d(s)\|\,ds.
\end{align*}
The first term tends to zero by \eqref{eq:uniform_window}.
For the other two terms, the Cauchy-Schwarz inequality gives
\begin{align*}
  \frac1{\ell_k}\int_{I_k}s\|\dot\lambda(s)\|\,ds
  &\leq
  \frac1{\ell_k}
  \left(\int_{I_k}s\,ds\right)^{1/2}
  \left(
    \int_{I_k}s\|\dot\lambda(s)\|^2\,ds
  \right)^{1/2}\\
  &=
  \frac{\sqrt{(\rho^2-1)/2}}{\rho-1}
  \left(
    \int_{I_k}s\|\dot\lambda(s)\|^2\,ds
  \right)^{1/2},\\
  \frac1{\ell_k}\int_{I_k}\|d(s)\|\,ds
  &\leq
  \frac1{\ell_k}
  \left(\int_{I_k}\frac{ds}{s}\right)^{1/2}
  \left(
    \int_{I_k}s\|d(s)\|^2\,ds
  \right)^{1/2}\\
  &=
  \frac{\sqrt{\log\rho}}{\ell_k}
  \left(
    \int_{I_k}s\|d(s)\|^2\,ds
  \right)^{1/2}.
\end{align*}
By \eqref{eq:strict_integrability} and $t_k\to+\infty$,
\[
  0\leq
  \int_{I_k}s\bigl(\|\dot\lambda(s)\|^2+\|d(s)\|^2\bigr)\,ds
  \leq
  \int_{t_k}^{+\infty}
  s\bigl(\|\dot\lambda(s)\|^2+\|d(s)\|^2\bigr)\,ds
  \longrightarrow0.
\]
Since $\rho>1$ is fixed and $\ell_k=(\rho-1)t_k\to+\infty$,
both right-hand sides above tend to zero,  this proves \eqref{eq:average_p}.

We next identify the limit of the averaged stationarity residual.
By \eqref{eq:uniform_window} and continuity,
\begin{equation}\label{eq_nablafto}
\begin{aligned}
&	 \sup_{s\in I_k}
  \|\nabla f(x(s))-\nabla f(\bar x)\|\to0,
  \\
 & \sup_{s\in I_k}
  \|J_g(x(s))-J_g(\bar x)\|\to0.
\end{aligned}
\end{equation}
Consequently, there exists $C_J>0$, independent of $k$,
such that
\[
  \sup_{s\in I_k}\|J_g(x(s))\|\leq C_J
\]
for all $k$. Using the identity
\[
  J_g(x(s))^\top p(s)-J_g(\bar x)^\top\bar\lambda
  =
  J_g(x(s))^\top(p(s)-\bar\lambda)
  +
  \bigl(J_g(x(s))-J_g(\bar x)\bigr)^\top\bar\lambda,
\]
and \eqref{eq:average_p}, we have
\begin{align*}
  &\left\|
    \frac1{\ell_k}\int_{I_k}
    J_g(x(s))^\top p(s)\,ds
    -J_g(\bar x)^\top\bar\lambda
  \right\|\\
  &\qquad\leq
  C_J\frac1{\ell_k}\int_{I_k}\|p(s)-\bar\lambda\|\,ds
  +
  \|\bar\lambda\|
  \sup_{s\in I_k}\|J_g(x(s))-J_g(\bar x)\|\\
    &\qquad\to0.
\end{align*}
Similarly, from \eqref{eq_nablafto},
\[
  \left\|
    \frac1{\ell_k}\int_{I_k}\nabla f(x(s))\,ds
    -\nabla f(\bar x)
  \right\|
  \leq
  \sup_{s\in I_k}
  \|\nabla f(x(s))-\nabla f(\bar x)\|
  \to0.
\]
Consequently,
\begin{equation}\label{eq_sa11}
\lim_{t_k\to+\infty}  \frac1{\ell_k}\int_{I_k}
  \bigl(\nabla f(x(s))+J_g(x(s))^\top p(s)\bigr)\,ds
  =   \nabla f(\bar x)+J_g(\bar x)^\top\bar\lambda.
\end{equation}

On the other hand, integrating the primal equation in
\eqref{eq:dy} over $I_k$ yields
\begin{equation}\label{eq_sa12}
\frac1{\ell_k}\int_{I_k}
  \bigl(\nabla f(x(s))+J_g(x(s))^\top p(s)\bigr)\,ds=
  -\frac{\dot x(\rho t_k)-\dot x(t_k)}{\ell_k}
  -\frac{\alpha}{\ell_k}
   \int_{I_k}\frac{\dot x(s)}s\,ds.
\end{equation}
By Theorem~\ref{thm:global}, there exists $C_1>0$ such that $\|\dot x(t)\|\leq C_1/t$
for all $t$. Hence
\[
  \frac{\|\dot x(\rho t_k)-\dot x(t_k)\|}{\ell_k}
  \leq
  \frac{C_1(1+\rho^{-1})}{(\rho-1)t_k^2},
  \qquad
  \frac1{\ell_k}\int_{I_k}\frac{\|\dot x(s)\|}{s}\,ds
  \leq
  \frac{C_1}{\rho t_k^2}.
\]
This, together with \eqref{eq_sa11} and \eqref{eq_sa12}, implies
\[
  \nabla f(\bar x)+J_g(\bar x)^\top\bar\lambda=0.
\]
Together with \eqref{eq:cluster_projection}, this gives
$\bar z=(\bar x,\bar \lambda)\in\Omega$.
\end{proof}

We now establish convergence of the  trajectory and the
improved rates.

\begin{theorem}\label{thm:trajectory_littleo}
Let $(x(t),\lambda(t))$ be the unique global solution of the dynamic
\eqref{eq:dy}. 
There exists $(x^\infty,\lambda^\infty)\in\Omega$ such that
\[
  (x(t),\lambda(t))
  \to(x^\infty,\lambda^\infty)
  \qquad\text{as }t\to+\infty.
\]
Moreover,
$
  \|\dot x(t)\|+\|\dot\lambda(t)\|=o(t^{-1})
$ and
\[
  \|\pos{g(x(t))}\|=o\left(\frac{1}{t^{2}}\right),
  \qquad
  |f(x(t))-f^\star|=o\left(\frac{1}{t^{2}}\right).
\]
\end{theorem}

\begin{proof}
We divide the proof into two steps.

\medskip
\noindent\emph{Step 1: convergence of the trajectory.}
For an arbitrary $z^\star=(x^\star,\lambda^\star)\in\Omega$,
define
\[
  h_{z^\star}(t)
  :=
  \frac12\|x(t)-x^\star\|^2
  +\frac1{2\sigma}\|\lambda(t)-\lambda^\star\|^2
\]
and
\begin{equation}\label{eq:defR}
  R_{z^\star}(t)
  :=
  t^2\Phi_{z^\star}(t)
  +\frac{t^2}{2}
  \left(
    \|\dot x(t)\|^2+\frac1\sigma\|\dot\lambda(t)\|^2
  \right),
\end{equation}
where $\Phi_{z^\star}(t)$ is defined in \eqref{eq_defd}.
Expanding the squares in \eqref{eq:E} and using
$\gamma+\delta=\alpha-1$, we obtain
\begin{equation}\label{eq:energy_distance}
  \mathcal E_{z^\star}(t)
  =
  R_{z^\star}(t)
  +\gamma t\dot h_{z^\star}(t)
  +\gamma(\alpha-1)h_{z^\star}(t).
\end{equation}
Furthermore, \eqref{eq:strict_integrability} gives
\begin{equation}\label{eq:trt}
\begin{aligned}
	 \int_{t_0}^{+\infty}\frac{R_{z^\star}(t)}t\,dt
  &=
  \int_{t_0}^{+\infty}t\Phi_{z^\star}(t)\,dt
  +\frac12\int_{t_0}^{+\infty}t
  \left(
    \|\dot x(t)\|^2+\frac1\sigma\|\dot\lambda(t)\|^2
  \right)\,dt\\
  &<+\infty.
\end{aligned}
\end{equation}
Rearranging \eqref{eq:energy_distance} gives
\[  t\dot h_{z^\star}(t)+(\alpha-1)h_{z^\star}(t)
  =
  \frac1\gamma\mathcal E_{z^\star}(t)
  -\frac1\gamma R_{z^\star}(t).
\]
Multiplying the identity by $t^{\alpha-2}$ and
integrating over $[t_0,t]$ gives
\begin{equation}\label{eq:hzteq}
\begin{aligned}
	 h_{z^\star}(t)
  =&
  \left(\frac{t_0}{t}\right)^{\alpha-1}  h_{z^\star}(t_0)
  +\frac1{\gamma t^{\alpha-1}}\int_{t_0}^t s^{\alpha-2}\mathcal E_{z^\star}(s)\,ds
  \\
  &-\frac1{\gamma t^{\alpha-1}}\int_{t_0}^t s^{\alpha-2} R_{z^\star}(s)\,ds.
\end{aligned}
\end{equation}
The first term tends to zero. For the last term, from \eqref{eq:trt}, for any $\varepsilon>0$, there exists $T>t_0$ such that
\[
\int_T^\infty
  \frac{R_{z^\star}(s)}s\,ds<\varepsilon.
\]
Since $\alpha>3$ and $R_{z^\star}\geq0$, for every $t\geq T$,
\begin{align*}
  \frac1{t^{\alpha-1}}
  \int_{t_0}^{t}s^{\alpha-2}R_{z^\star}(s)\,ds&=
  \frac1{t^{\alpha-1}}
  \int_{t_0}^{T}s^{\alpha-2}R_{z^\star}(s)\,ds
  +\int_T^t
  \left(\frac{s}{t}\right)^{\alpha-1}
  \frac{R_{z^\star}(s)}s\,ds\\
  &\leq
  \frac1{t^{\alpha-1}}
  \int_{t_0}^{T}s^{\alpha-2}R_{z^\star}(s)\,ds
  +\int_T^\infty\frac{R_{z^\star}(s)}s\,ds\\
  &\leq
  \frac1{t^{\alpha-1}}
  \int_{t_0}^{T}s^{\alpha-2}R_{z^\star}(s)\,ds
  +\varepsilon.
\end{align*}
 Since $T$ is fixed, taking the upper limit as $t\to+\infty$ gives
\[
  0\leq
  \limsup_{t\to+\infty}
  \frac1{t^{\alpha-1}}
  \int_{t_0}^{t}s^{\alpha-2}R_{z^\star}(s)\,ds
  \leq\varepsilon.
\]
As $\varepsilon>0$ is arbitrary, we conclude that
\begin{equation}\label{eq:Rlim}
  \lim_{t\to+\infty}
  \frac1{t^{\alpha-1}}
  \int_{t_0}^{t}s^{\alpha-2}R_{z^\star}(s)\,ds
  =0.
\end{equation}
By \eqref{eq:energy_limit}, let
$\mathcal E_\infty=\lim_{t\to+\infty}\mathcal E_{z^\star}(t)$. L'Hôpital's rule gives
\[
  \lim_{t\to+\infty}
  \frac1{t^{\alpha-1}}
  \int_{t_0}^t s^{\alpha-2}\mathcal E_{z^\star}(s)\,ds
  =\frac{\mathcal E_\infty}{\alpha-1}.
\]
Combining this limit with  \eqref{eq:hzteq} and
\eqref{eq:Rlim}  yields
\begin{equation}\label{eq:hlim}
  \lim_{t\to+\infty}h_{z^\star}(t)
  = \lim_{t\to+\infty}
  \frac{\int_{t_0}^t s^{\alpha-2}\mathcal E_{z^\star}(s)\,ds}{\gamma t^{\alpha-1}} 
=
  \frac{\mathcal E_\infty}{\gamma(\alpha-1)}
\end{equation}
Thus $h_{z^\star}(t)$ has a finite limit. 

 We apply Lemma~\ref{lem:opial} on
$\R^n\times\R^m$ equipped with the weighted norm
\[
  \|(x,\lambda)\|_\sigma
  :=
  \left(\|x\|^2+\frac1\sigma\|\lambda\|^2\right)^{1/2}.
\]
For every $z^\star\in\Omega$, $h_{z^\star}(t)$ has a finite limit. Since
\[
  \|z(t)-z^\star\|_\sigma
  =
  \sqrt{2h_{z^\star}(t)},
\]
the first condition of Lemma~\ref{lem:opial} is satisfied. Moreover, every cluster point belongs to $\Omega$
by Lemma~\ref{lem:cluster}.
Hence Lemma~\ref{lem:opial} yields
\[
  (x(t),\lambda(t))
  \to (x^\infty,\lambda^\infty)
  \quad\text{for some }(x^\infty,\lambda^\infty)\in\Omega.
\]

\medskip
\noindent\emph{Step 2: improved convergence rates.}
Take $z^\star=z^\infty$. By Theorem~\ref{thm:global},
$t\dot x(t)$ and $t\dot\lambda(t)$ are bounded.
Together with the convergence established in Step~1,
this gives
\[
  t\dot h_{z^\infty}(t)
  =
  \ip{x(t)-x^\infty}{t\dot x(t)}
  +\frac1\sigma
  \ip{\lambda(t)-\lambda^\infty}{t\dot\lambda(t)}
  \to0.
\]
Since $h_{z^\infty}(t)\to0$, equation \eqref{eq:hlim} yields
$\mathcal E_{z^\infty}(t)\to0$. Therefore, \eqref{eq:energy_distance}
implies
\begin{equation}\label{eq:R_to_zero}
  R_{z^\infty}(t)
  =
  \mathcal E_{z^\infty}(t)
  -\gamma t\dot h_{z^\infty}(t)
  -\gamma(\alpha-1)h_{z^\infty}(t)
  \to0.
\end{equation}
Since $z^\infty\in\Omega$,
$\mathcal L(x^\infty,\lambda^\infty)=f^\star$.
Moreover, both terms in \eqref{eq:defR} are nonnegative.
Hence
\[
  0\leq
  t^2\bigl(
    \mathcal L(x(t),\lambda^\infty)-f^\star
  \bigr)
  \leq R_{z^\infty}(t)
\]
and
\[
  0\leq
  \frac{t^2}{2}
  \left(
    \|\dot x(t)\|^2
    +\frac1\sigma\|\dot\lambda(t)\|^2
  \right)
  \leq R_{z^\infty}(t).
\]
Therefore, from \eqref{eq:R_to_zero},
\begin{equation}\label{eq:o_velocity}
  \mathcal L(x(t),\lambda^\infty)-f^\star=o(t^{-2}),\qquad
  \|\dot x(t)\|+\|\dot\lambda(t)\|=o(t^{-1}).
\end{equation}

We next establish the feasibility estimate.
As in the proof of Theorem~\ref{thm:rates}, set
$G(t)=t^2g(x(t))$. Equation~\eqref{eq:Volterra} gives
\[
  G(t)+(\gamma-2)\int_{t_0}^t\frac{G(s)}s\,ds
  =B(t)-M(t),
\]
where
$M(t)$
and
$B(t)$ are defined in \eqref{eq:defMB}. By Step~1 and \eqref{eq:o_velocity},
\[
  \widehat\lambda(t)
  =
  \lambda(t)+\frac{t}{\gamma}\dot\lambda(t)
  \to\lambda^\infty,
\]
so $B(t)$ has a finite limit. Since $\gamma-2>0$ and
$M$ is componentwise nondecreasing with $M(t_0)=0$,
Lemma~\ref{lem:volterra}(ii) yields
\[
 \lim_{t\to+\infty} t^2\|\pos{g(x(t))}\|
  =
  \lim_{t\to+\infty} \|\pos{G(t)}\|
  =0.
\]
Thus \[\|\pos{g(x(t))}\|=o(t^{-2}).\]

Finally, we prove the objective residual estimate.
The trajectory convergence,
\eqref{eq:o_velocity}, and
\eqref{eq:hatL}-\eqref{eq:hatg} give
\[
  \lim_{t\to+\infty}\widehat\lambda(t)=\lambda^\infty,
  \qquad
  \lim_{t\to+\infty}\widehat g(t)=g(x^\infty).
\]
For each component with $\lambda_i^\infty>0$,
\eqref{eq:KKT} gives $g_i(x^\infty)=0$. Hence
\[
  \lim_{t\to+\infty}
  \bigl(\widehat\lambda_i(t)+\beta\widehat g_i(t)\bigr)
  =\lambda_i^\infty>0.
\]
By \eqref{eq:eta}, there exists $T_i\geq t_0$ such that
$\eta_i(t)=0$ for all $t\geq T_i$. The definition of $M$ in \eqref{eq:defMB}
then gives
\[
  M_i(t)-M_i(T_i)
  =
  \gamma\int_{T_i}^t s\eta_i(s)\,ds
  =0,\qquad t\geq T_i.
\]
Since $B(t)$ has a finite limit and $\gamma>2$,
Lemma~\ref{lem:volterra}(ii), applied to
\eqref{eq:Volterra} with $c=\gamma-2$, yields
\[  \lim_{t\to+\infty}G_i(t)=0
  \qquad\text{whenever }\lambda_i^\infty>0.
\]
Consequently,
\[
  \lim_{t\to+\infty}\ip{\lambda^\infty}{G(t)}
  =
  \sum_{\{i:\lambda_i^\infty>0\}}
  \lambda_i^\infty\lim_{t\to+\infty}G_i(t)
  =0.
\]
Since $G(t)=t^2g(x(t))$, combining this limit with \eqref{eq:o_velocity}
and the identity
\[
  t^2\bigl(f(x(t))-f^\star\bigr)
  =
  t^2\bigl(\mathcal L(x(t),\lambda^\infty)-f^\star\bigr)
  -\ip{\lambda^\infty}{G(t)},
\]
we obtain
\[\lim_{t\to+\infty}t^2(f(x(t))-f^\star)=0.\]
This proves 
\[|f(x(t))-f^\star|=o(t^{-2})\] and completes
the proof.
\end{proof}

\begin{remark}
Theorem~\ref{thm:trajectory_littleo} complements the
convergence results for unconstrained inertial dynamics
\cite{AttouchMathProg,MayTjm} and accelerated primal-dual dynamics with linear
constraints \cite{BotJde,HeLittleO}. In the affine case, the substitution
$\theta=1/\gamma$ turns $2<\gamma<\alpha-1$ into the strict
extrapolation range $1/(\alpha-1)<\theta<1/2$ used in
\cite{HeLittleO}; thus the parameter regime is consistent with the 
linear constraint results. 
 \end{remark}

 
\section{Extensions with perturbations and scaling}
\label{sec:extensions}
We retain the standing assumptions on the optimization problem and the
constant parameters in \eqref{eq:para}. The two extensions below are
considered separately, and the extrapolated quantities
$\widehat\lambda$, $\widehat g$, and $p$ remain defined by
\eqref{eq:hatL}--\eqref{eq:p}.

\subsection{Additive perturbations in the primal equation}
\label{sec:perturbation}
Let $\epsilon:[t_0,+\infty)\to\R^n$ be a continuous vector-valued
perturbation satisfying
\begin{equation}\label{eq:perturb_assumption}
  \int_{t_0}^{+\infty}t\|\epsilon(t)\|\,dt<+\infty.
\end{equation}
We consider the perturbed dynamic
\begin{equation}\label{eq:dy_perturbed}
\left\{
\begin{aligned}
  \ddot x(t)+\frac\alpha t\dot x(t)
    +\nabla f(x(t))+J_g(x(t))^\top p(t)&=\epsilon(t),\\
  \ddot\lambda(t)+\frac\alpha t\dot\lambda(t)
    -\frac\sigma\beta\bigl(p(t)-\widehat\lambda(t)\bigr)&=0.
\end{aligned}
\right.
\end{equation}
Thus the perturbation acts only on the primal equation, while the dual
equation and the definitions of $\widehat\lambda$, $\widehat g$, and $p$
are unchanged. For example, \eqref{eq:perturb_assumption} holds whenever
$\|\epsilon(t)\|\leq C t^{-2-\nu}$ for some $C>0$ and $\nu>0$.

\begin{proposition}\label{prop:perturbed_extension}
Under \eqref{eq:perturb_assumption}, for every initial point the dynamic
\eqref{eq:dy_perturbed} has a unique global solution. The primal-dual
trajectory is bounded, and
\[
  \|\dot x(t)\|+\|\dot\lambda(t)\|=\mathcal O(t^{-1}),
  \qquad
  \|\pos{g(x(t))}\|+|f(x(t))-f^\star|=\mathcal O(t^{-2}).
\]
If $2<\gamma<\alpha-1$, the trajectory converges to a KKT pair,
the velocities are $o(t^{-1})$, and both residuals are $o(t^{-2})$.
\end{proposition}

\begin{proof}
Local existence and uniqueness follow as in Proposition~\ref{prop:local},
since the added forcing is continuous and independent of the state.
Fix $z^\star=(x^\star,\lambda^\star)\in\Omega$ and use the energy
$\mathcal E_{z^\star}$ in \eqref{eq:E}. With
$Y_x(t)=\gamma(x(t)-x^\star)+t\dot x(t)$, the calculation leading to
\eqref{eq:Edot} gives
\begin{align}\label{eq:perturbed_energy}
  \dot{\mathcal E}_{z^\star}(t)
  \leq{}&-(\gamma-2)t\Phi_{z^\star}(t)
  -\delta t\left(\|\dot x(t)\|^2
       +\frac1\sigma\|\dot\lambda(t)\|^2\right)
  -\frac{\gamma t}{\beta}\|d(t)\|^2\notag\\
  &+t\ip{Y_x(t)}{\epsilon(t)},
\end{align}
where $\Phi_{z^\star}$ and $d$ are defined in \eqref{eq_defd}.
Since $\|Y_x(t)\|\leq\sqrt{2\mathcal E_{z^\star}(t)}$, applying this
inequality to $\sqrt{\mathcal E_{z^\star}+a}$ and then letting
$a\downarrow0$ yields
\[
  \sqrt{\mathcal E_{z^\star}(t)}
  \leq\sqrt{\mathcal E_{z^\star}(t_0)}
  +\frac1{\sqrt2}\int_{t_0}^t s\|\epsilon(s)\|\,ds.
\]
The energy is therefore bounded. Lemma~\ref{le_bounde} and the
continuation argument in Theorem~\ref{thm:global} give boundedness of
the trajectory, the velocity estimate, and global existence.

Because the dual equation is unchanged, the identities
\eqref{eq:Volterra} and \eqref{eq:defMB} remain valid with
$G(t)=t^2g(x(t))$. Boundedness of $\lambda$ and $\widehat\lambda$
implies boundedness of $B$, so Lemma~\ref{lem:volterra}(i) gives the
feasibility rate. For the objective upper bound, use
$\mathcal E_0=t^2(f(x)-f^\star)+Q$ from \eqref{eq:E0}.
The state and rescaled velocity bounds give constants $K,C>0$ such that
$0\leq Q\leq K$ and
\[
  \dot{\mathcal E}_0+\frac{\gamma-2}{t}\mathcal E_0
  \leq\frac{(\gamma-2)K}{t}+Ct\|\epsilon(t)\|.
\]
An integrating factor gives
\[
  \mathcal E_0(t)
  \leq\max\{K,\mathcal E_0(t_0)\}
    +C\int_{t_0}^t s\|\epsilon(s)\|\,ds.
\]
Together with the saddle-point lower bound \eqref{eq:lower}, this proves
the absolute objective rate.

In the strict regime, boundedness of $Y_x$ and
\eqref{eq:perturb_assumption} make the perturbation work in
\eqref{eq:perturbed_energy} absolutely integrable. It follows that
$\mathcal E_{z^\star}$ has a finite limit and all three integrals in
\eqref{eq:strict_integrability} are finite. The proof of
Lemma~\ref{lem:cluster} still applies. Indeed, for
$I_k=[t_k,\rho t_k]$ and $\ell_k=(\rho-1)t_k$, the additional term in
the averaged primal equation satisfies
\[
  \left\|\frac1{\ell_k}\int_{I_k}\epsilon(s)\,ds\right\|
  \leq\frac1{(\rho-1)t_k^2}
      \int_{t_k}^{\rho t_k}s\|\epsilon(s)\|\,ds
  \longrightarrow0.
\]
The energy-distance identity \eqref{eq:energy_distance} is unchanged.
Thus the proof of Theorem~\ref{thm:trajectory_littleo} gives trajectory
convergence, $\mathcal E_{z^\infty}\to0$, and the velocity estimate.
The function $B$ in \eqref{eq:defMB} then has a finite limit, so
Lemma~\ref{lem:volterra}(ii) gives the little-$o$ feasibility rate.
For components with $\lambda_i^\infty>0$, eventual positivity of
$\widehat\lambda_i+\beta\widehat g_i$ makes $\eta_i$ vanish and $M_i$
eventually constant. The last assertion of that lemma then gives the
little-$o$ objective rate, as in the proof of
Theorem~\ref{thm:trajectory_littleo}.
\end{proof}

\subsection{Time-dependent scaling coefficients}
\label{sec:scaling}
Let $b\in C^1([t_0,+\infty);(0,+\infty))$ satisfy
\begin{equation}\label{eq:scaling_assumption}
  0\leq t\dot b(t)\leq(\gamma-2)b(t),\qquad t\geq t_0.
\end{equation}
We multiply both primal and dual coupling terms by the same coefficient:
\begin{equation}\label{eq:dy_scaled}
\left\{
\begin{aligned}
  \ddot x+\frac\alpha t\dot x
    +b(t)\bigl(\nabla f(x)+J_g(x)^\top p\bigr)&=0,\\
  \ddot\lambda+\frac\alpha t\dot\lambda
    -\frac{\sigma b(t)}\beta(p-\widehat\lambda)&=0.
\end{aligned}
\right.
\end{equation}
The use of a common coefficient preserves the cancellation between
the primal and dual terms in the energy estimate.

\begin{proposition}\label{prop:scaled_extension}
Under \eqref{eq:scaling_assumption}, the dynamic \eqref{eq:dy_scaled}
has a unique global solution with a bounded primal-dual trajectory and
$\|\dot x(t)\|+\|\dot\lambda(t)\|=\mathcal O(t^{-1})$. Moreover,
\begin{equation}\label{eq:scaled_rates}
  \|\pos{g(x(t))}\|+|f(x(t))-f^\star|
  =\mathcal O\!\left(\frac1{t^2b(t)}\right).
\end{equation}
In particular, for
\begin{equation}\label{eq:power_scaling}
  b(t)=b_0(t/t_0)^q,\qquad b_0>0,\quad 0\leq q\leq\gamma-2,
\end{equation}
both residuals are $\mathcal O(t^{-2-q})$.
If $2<\gamma<\alpha-1$ and $0\leq q<\gamma-2$, the trajectory converges
to a KKT pair, the velocities are $o(t^{-1})$, and both residuals are
$o(t^{-2-q})$.
\end{proposition}

\begin{proof}
For $z^\star\in\Omega$, replace only the first term in \eqref{eq:E} by
$t^2b(t)\Phi_{z^\star}(t)$, and denote the resulting energy by
$\mathcal E^b_{z^\star}$. Set
\[
  a_b(t):=\frac{\gamma-2}{t}-\frac{\dot b(t)}{b(t)}\geq0.
\]
The calculation leading to \eqref{eq:Edot} now gives
\begin{equation}\label{eq:scaled_energy}
  \dot{\mathcal E}^b_{z^\star}
  \leq-a_b(t)t^2b(t)\Phi_{z^\star}
  -\delta t\left(\|\dot x\|^2+\frac1\sigma\|\dot\lambda\|^2\right)
  -\frac{\gamma tb(t)}\beta\|d\|^2\leq0.
\end{equation}
Thus Lemma~\ref{le_bounde} and the continuation argument give global
well-posedness and the stated bounds. Since $b(t)\geq b(t_0)>0$, no
degeneracy of the energy occurs.

Set $G_b=t^2b(t)g(x(t))$. The scaled dual equation yields
\begin{equation}\label{eq:scaled_volterra}
  G_b(t)+\int_{t_0}^t a_b(s)G_b(s)\,ds=B_b(t)-M_b(t),
\end{equation}
where
\begin{align*}
  B_b(t)&=t_0^2b(t_0)g(x(t_0))
  +\frac{\gamma^2}\sigma
    \bigl(\widehat\lambda(t)-\widehat\lambda(t_0)\bigr)
  +\frac{\gamma\delta}\sigma
    \bigl(\lambda(t)-\lambda(t_0)\bigr),\\
  M_b(t)&=\gamma\int_{t_0}^t sb(s)\eta(s)\,ds,
\end{align*}
and $\eta$ is defined in \eqref{eq:eta}. The function $B_b$ is bounded
and $M_b$ is componentwise nondecreasing. Lemma~\ref{lem:volterra}(i)
therefore gives the feasibility estimate in \eqref{eq:scaled_rates}.
For the objective upper bound, use
$\mathcal E^b_0=t^2b(t)(f(x)-f^\star)+Q$, with $Q$ from
\eqref{eq:E0}. As before, $0\leq Q\leq K$ for some $K>0$, and
\[
  \dot{\mathcal E}^b_0+a_b(t)\mathcal E^b_0\leq a_b(t)K.
\]
Hence $\mathcal E^b_0\leq\max\{K,\mathcal E^b_0(t_0)\}$.
Together with the saddle-point lower bound, this proves
\eqref{eq:scaled_rates}.

For \eqref{eq:power_scaling} with $q<\gamma-2$ and
$2<\gamma<\alpha-1$, we have $a_b(t)=(\gamma-2-q)/t>0$.
Equation~\eqref{eq:scaled_energy} implies
\[
  \int_{t_0}^{+\infty}tb(t)\Phi_{z^\star}(t)\,dt<+\infty,
  \qquad
  \int_{t_0}^{+\infty}
    t\bigl(\|\dot z(t)\|^2+\|d(t)\|^2\bigr)\,dt<+\infty.
\]
The cluster-point argument remains valid. In particular, on
$I_k=[t_k,\rho t_k]$ with $\ell_k=(\rho-1)t_k$, integration by parts in
the scaled primal equation gives
\[
  \frac1{\ell_k}\int_{I_k}
    \bigl(\nabla f(x(s))+J_g(x(s))^\top p(s)\bigr)\,ds
  =-\frac1{\ell_k}\left[\frac{\dot x(s)}{b(s)}\right]_{t_k}^{\rho t_k}
   -\frac{\alpha+q}{\ell_k}\int_{I_k}\frac{\dot x(s)}{sb(s)}\,ds
  \longrightarrow0.
\]
Here $b\geq b_0$ and $\dot x=\mathcal O(t^{-1})$ justify the limit.
Define
\[
  R^b_{z^\star}(t)=t^2b(t)\Phi_{z^\star}(t)
  +\frac{t^2}{2}
    \left(\|\dot x(t)\|^2+\frac1\sigma\|\dot\lambda(t)\|^2\right).
\]
Then \eqref{eq:energy_distance} holds with
$(\mathcal E_{z^\star},R_{z^\star})$ replaced by
$(\mathcal E^b_{z^\star},R^b_{z^\star})$, and
$\int_{t_0}^{+\infty}R^b_{z^\star}(t)/t\,dt<+\infty$.
The proof of Theorem~\ref{thm:trajectory_littleo} yields trajectory
convergence, $R^b_{z^\infty}\to0$, and the velocity estimate.
Finally, $B_b$ has a finite limit. Applying
Lemma~\ref{lem:volterra}(ii) to \eqref{eq:scaled_volterra} with
$c=\gamma-2-q$ gives $\|\pos{G_b(t)}\|\to0$.
For $\lambda_i^\infty>0$, eventual positivity of
$\widehat\lambda_i+\beta\widehat g_i$ makes $\eta_i$ vanish, so
$M_{b,i}$ is eventually constant and $G_{b,i}\to0$.
Therefore
\[
  t^2b(t)(f(x(t))-f^\star)
  =t^2b(t)\Phi_{z^\infty}(t)-\ip{\lambda^\infty}{G_b(t)}
  \longrightarrow0.
\]
This proves the claimed little-$o$ rates.
\end{proof}

\section{Numerical experiments}\label{sec:numerics}
We compare the proposed accelerated inertial dynamic \eqref{eq:dy} with the noninertial
augmented primal-dual gradient dynamic of Tang, Qu, and Li
\cite{Tang2020}:
\begin{equation}\label{dy:noninertial}
\left\{
\begin{aligned}
 \dot x(t)&=-\nabla f(x(t))-J_g(x(t))^\top [\lambda(t)+\beta g(x(t))]_+,\\
 \dot\lambda(t)&=\frac{\sigma}{\beta}\bigl([\lambda(t)+\beta g(x(t))]_+-\lambda(t)\bigr).\end{aligned}
\right.
\end{equation}
Both dynamics use the same augmentation parameter. We measure their performance in the continuous-time variable $t$ using
\[
 r_f(t)=|f(x(t))-f^\star|,\qquad
 r_g(t)=\|[g(x(t))]_+\|,\qquad R(t)=r_f(t)+r_g(t).
\]
For the proposed inertial dynamic \eqref{eq:dy}, we also plot
$t^2R(t)$ and $t(\|\dot x(t)\|+\|\dot\lambda(t)\|)$ to illustrate
the little-$o$ estimates.

\emph{Example 1}: Consider the convex problem
\begin{equation}\label{eq:numerics:flat5d}
\begin{aligned}
 \min_{x\in\mathbb R^5}\quad & f(x)=\frac13\sum_{i=1}^4|x_i|^3\\
 \text{subject to}\quad
 &(x_1-1)^2+x_2^4\leq1,\qquad
   (x_3-1)^2+x_4^4\leq1,\qquad x_5^2\leq1.
\end{aligned}
\end{equation}
The optimal value is $f^\star=0$, and the set of KKT pairs is
$
 \Omega=\bigl\{((0,0,0,0,s)^\top,0): |s|\leq1\bigr\}.
$
We set $\alpha=5$, $\gamma=3$, $\beta=0.01$, and $\sigma=2$,
so that $2<\gamma<\alpha-1$.
Both dynamics are integrated over $t\in[1,200]$, with relative and
absolute tolerances $10^{-8}$ and $10^{-10}$, respectively, and
a maximum integration step of $0.1$.
Figure~\ref{fig:numerics:flat5d} shows that the inertial dynamic \eqref{eq:dy} yields a smaller residual than the
noninertial dynamic \eqref{dy:noninertial}. Both scaled inertial quantities decrease toward
zero over the displayed interval, consistent with
$R(t)=o(t^{-2})$ and
$\|\dot x(t)\|+\|\dot\lambda(t)\|=o(t^{-1})$.
The component plots show the first four primal coordinates approaching
zero, the fifth settling at a value in $[-1,1]$, and the multipliers
approaching zero, illustrating convergence to a KKT pair.

\begin{figure}[t]
 \centering
 \begin{tabular}{@{}c@{\hspace{0.02\textwidth}}c@{}}
  \includegraphics[width=0.47\textwidth]{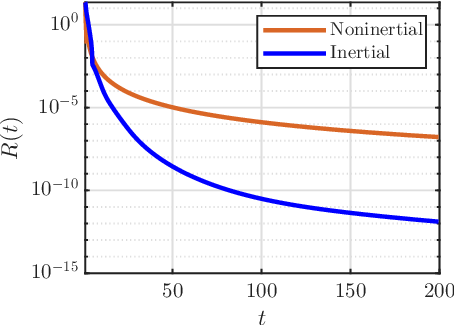}&
  \includegraphics[width=0.47\textwidth]{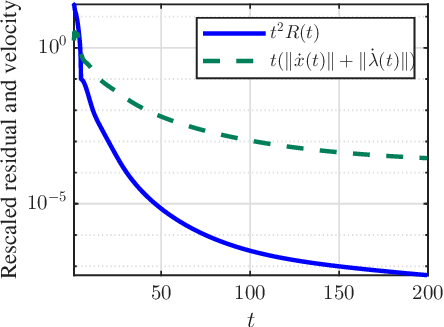}\\[-1mm]
  {\small (a) Combined residual.}&
  {\small (b) Rescaled residual and velocity.}\\[1mm]
  \includegraphics[width=0.47\textwidth]{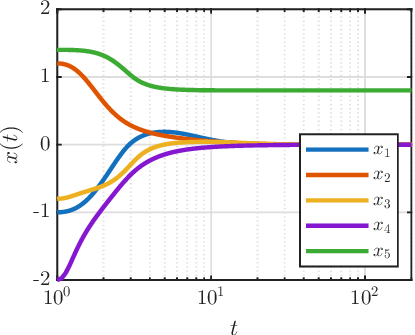}&
  \includegraphics[width=0.47\textwidth]{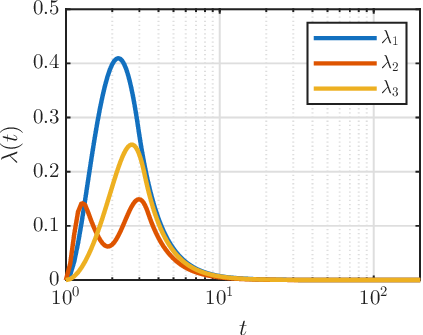}\\[-1mm]
  {\small (c) Primal trajectory.}&
  {\small (d) Dual trajectory.}
 \end{tabular}
\caption{Residuals and inertial trajectories for problem~\eqref{eq:numerics:flat5d}.}
 \label{fig:numerics:flat5d}
\end{figure}
 
\emph{Example 2}:
We consider the basis pursuit denoising model tested in
\cite[Sections~5.1 and~6.1]{XuLinear2021}:
\[
 \min_{z\in\mathbb R^p}\|z\|_1
 \qquad\text{subject to}\qquad \|Az-b\|\leq\delta.
\]
Introducing an epigraph variable $u\in\mathbb R^p$ yields the equivalent
smooth convex formulation \cite{CandesE}:
\begin{equation}\label{eq:numerics:bpdn}
\begin{aligned}
 \min_{(z,u)\in\mathbb R^{2p}}\quad &f(z,u)=\boldsymbol{1}^{\top}u\\
 \text{subject to}\quad
 &g_0(z,u)={\|Az-b\|^2-\delta^2}\leq0,\\
 &z-u\leq0,\qquad -z-u\leq0.
\end{aligned}
\end{equation}
The objective is linear, and all constraint functions are smooth and
convex. We generate a new Gaussian instance with $p=500$ and
$A\in\mathbb R^{100\times500}$, whose entries are independently drawn
from $\mathcal N(0,1/100)$. Thus \eqref{eq:numerics:bpdn} has
$1000$ primal variables and $1001$ inequalities.
We set $b=Az^{\rm true}+\xi$, where $z^{\rm true}$ has ten nonzero
entries, and choose $\|\xi\|=0.05\|Az^{\rm true}\|$ and
$\delta=\|\xi\|$.  We set $\alpha=20$, $\gamma=18$, $\beta=0.3$, and $\sigma=2$,
and integrate both dynamics over $t\in[1,500]$.
Here $x=(z,u)$, so the velocity norm includes both primal blocks.
Figure~\ref{fig:numerics:bpdn} illustrates the improved residual decay
of the inertial dynamics, together with the decay of both rescaled quantities toward zero, consistent
with the theoretical little-$o$ estimates.

\begin{figure}[t]
 \centering
 \begin{tabular}{@{}c@{\hspace{0.02\textwidth}}c@{}}
  \includegraphics[width=0.47\textwidth]{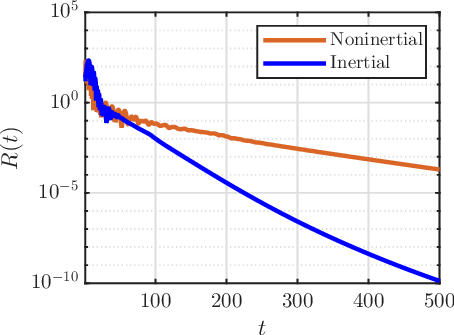}&
  \includegraphics[width=0.47\textwidth]{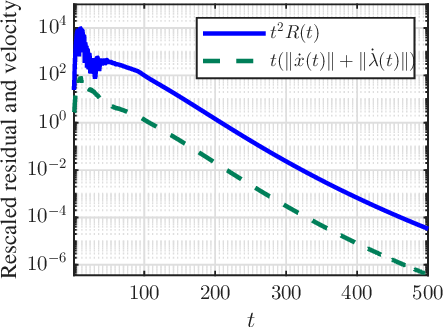}\\[-1mm]
  {\small (a) Combined residual.}&
  {\small (b) Rescaled inertial quantities.}
 \end{tabular}
 \caption{Residuals and rescaled gap for
problem~\eqref{eq:numerics:bpdn}.}
 \label{fig:numerics:bpdn}
\end{figure}

\section{Conclusions}\label{sec:conclusion}

This paper shows that Nesterov-type acceleration can be extended
to primal-dual multiplier dynamics for the nonlinear inequality constrained
problem \eqref{ques_main}. The key is to combine tangent constraint extrapolation with a nonnegative projected multiplier, which preserves the
coupling estimates needed in the Lyapunov analysis despite the
nonlinearity of the constraints. The resulting guarantees concern both residual decay and the limiting behavior of the state: under the strict parameter
conditions, the primal-dual trajectory converges to a KKT
pair, while the objective and feasibility residuals improve from
$O(t^{-2})$ to $o(t^{-2})$. These conclusions require neither strong
convexity nor uniqueness of the solution or multiplier. The present analysis is developed in continuous time. A natural next step is to construct practical discretizations that preserve the primal-dual coupling and provide accelerated estimates for both objective error and constraint violation.

\appendix
\section{Technical lemmas}\label{sec:appendix}
The continuous-time auxiliary results are collected here.

\begin{lemma}\label{lem:proj}
Let $C\subseteq\R^m$ be a nonempty closed convex set. For every
$z,y\in\R^m$,
\[
  y=\proj_C(z)
  \quad\Longleftrightarrow\quad
  y\in C,\qquad
  \ip{z-y}{v-y}\leq0
  \quad\forall\,v\in C.
\]
Moreover, for every $\theta>0$,
\[
  y=\proj_C
   \bigl(\theta z+(1-\theta)y\bigr)
  \quad\Longleftrightarrow\quad
  y=\proj_C(z).
\]
\end{lemma}

\begin{proof}
The first equivalence is the standard variational characterization
of metric projections; see
\cite[Theorem~3.16]{Bauschke16}. Since
\[
  \theta z+(1-\theta)y-y=\theta(z-y)
\]
and $\theta>0$, applying this characterization proves the second equivalence.
\end{proof}

The following lemma is a version of
\cite[Lemma~1]{HeHeGuo} on $[t_0,T)$. Although the original result is
stated on $[t_0,+\infty)$, the same argument applies when
$T\in(t_0,+\infty]$.
\begin{lemma}\label{le_bounde}
Let $T\in(t_0,+\infty]$. Assume that
$x:[t_0,T)\to\R^n$ and
$a:[t_0,T)\to(0,+\infty)$ are continuously differentiable,
$x^\star\in\R^n$, and let $C>0$. If
\[
  \|x(t)-x^\star+a(t)\dot x(t)\|^2\leq C,\qquad t\in[t_0,T),
\]
then $x(t)$ is bounded on $[t_0,T)$ and
$
  \sup_{t\in[t_0,T)}a(t)\|\dot x(t)\|<+\infty.
$
\end{lemma}
\begin{proof}
Set $h(t)=\|x(t)-x^\star\|^2$. Expanding the square
in the assumption yields
\[
  h(t)+a(t)\dot h(t)+a(t)^2\|\dot x(t)\|^2\leq C,
\]
and hence
\[
  \dot h(t)+\frac{1}{a(t)}h(t)\leq\frac{C}{a(t)}.
\]
Multiplying by the integrating factor
$\exp\bigl(\int_{t_0}^{t}1/a(s)\,ds\bigr)$ and integrating, 
we obtain
\[
  h(t)\leq
  C+\bigl(h(t_0)-C\bigr)
  \exp\left(-\int_{t_0}^{t}\frac{ds}{a(s)}\right)
  \leq\max\{h(t_0),C\}.
\]
Thus $x$ is bounded on $[t_0,T)$. Moreover,
\[
  a(t)\|\dot x(t)\|
  \leq
  \|x(t)-x^\star+a(t)\dot x(t)\|
  +\|x(t)-x^\star\|
  \leq
  \sqrt C+\max\{\|x(t_0)-x^\star\|,\sqrt C\},
\]
which proves the second assertion.
\end{proof}

The following result is a continuous-time variant of
Opial's lemma; see \cite{Bauschke16}.
\begin{lemma}\label{lem:opial}
Let $S\subset\R^d$ be nonempty, let
$u:[t_0,+\infty)\to\R^d$ be bounded, and let $\|\cdot\|$ be any
fixed norm on $\R^d$. Suppose that:
\begin{itemize}
  \item[(i)] For every $v\in S$,
  $\lim_{t\to+\infty}\|u(t)-v\|$ exists and is finite.
  \item[(ii)] Every cluster point of $u(t)$ as
  $t\to+\infty$ belongs to $S$.
\end{itemize}
Then there exists $u^\infty\in S$ such that
$u(t)\to u^\infty$ as $t\to+\infty$.
\end{lemma}

\begin{lemma}
\label{lem:volterra}
Let  $t_0>0$. Let $a:[t_0,+\infty)\to[0,+\infty)$ be continuous,
$G,B:[t_0,+\infty)\to\R^m$ be continuous, and
$M:[t_0,+\infty)\to\R_+^m$ be componentwise nondecreasing with
$M(t_0)=0$. Suppose that
\begin{equation}\label{eq:vol_rel}
  G(t)
  +
  \int_{t_0}^t a(s)G(s)\,ds
  =
  B(t)-M(t), \qquad \forall t\geq t_0.
\end{equation}
Then the following statements hold.
\begin{itemize}
  \item[(i)] If
  $\sup_{t\geq t_0}\norm{B(t)}\leq C$ for some $C>0$, then
\[
    \sup_{t\geq t_0}\norm{\pos{G(t)}}\leq2C.
\]
  \item[(ii)] Suppose that $a(t)=c/t$ for some $c>0$ and that $B(t)$ has a finite limit $B_\infty\in\R^m$ as $t\to+\infty$. Then
\[
    \lim_{t\to+\infty}\norm{\pos{G(t)}}=0.
\]   Moreover, if, for some component $i$, there exists
$T_i\geq t_0$ such that
$
  M_i(t)=M_i(T_i)$ for all $t\geq T_i$,
then
\[
  \lim_{t\to+\infty}G_i(t)=0.
\]
\end{itemize}
\end{lemma}

\begin{proof}
Set
$
  Z(t):=\int_{t_0}^t a(s)G(s)\,ds.
$
Then \eqref{eq:vol_rel} gives
\[
  \dot Z(t)+a(t)Z(t)
  =
  a(t)\bigl(B(t)-M(t)\bigr),
  \qquad
  Z(t_0)=0.
\]
Multiplying both sides by the integrating factor
$
  e^{\int_{t_0}^t a(\tau)\,d\tau},
$
gives
\[
  \frac{d}{dt}
  \left[
    e^{\int_{t_0}^t a(\tau)\,d\tau}Z(t)
  \right]
  =
  e^{\int_{t_0}^t a(\tau)\,d\tau}
  a(t)\bigl(B(t)-M(t)\bigr).
\]
Integrating from $t_0$ to $t$ and using $Z(t_0)=0$, we obtain
\begin{align}\label{eq_ztpara}
  Z(t)
  =
  e^{-\int_{t_0}^t a(\tau)\,d\tau}
  \left(\int_{t_0}^t
    e^{\int_{t_0}^s a(\tau)\,d\tau}
    a(s)B(s)\,ds-\int_{t_0}^t
    e^{\int_{t_0}^s a(\tau)\,d\tau}
    a(s)M(s)\,ds\right).
\end{align}
Since $M$ is componentwise nondecreasing, $0\leq M(s)\leq M(t)$ componentwise for $s\leq t$. Because $a(s)\geq0$,
\begin{align*}
  0
  &\leq
  e^{-\int_{t_0}^t a(\tau)\,d\tau}
  \int_{t_0}^t
    e^{\int_{t_0}^s a(\tau)\,d\tau}
    a(s)M(s)\,ds
  \\
  &\leq
  M(t)
  e^{-\int_{t_0}^t a(\tau)\,d\tau}
  \int_{t_0}^t
    e^{\int_{t_0}^s a(\tau)\,d\tau}
    a(s)\,ds
  \\
  &=
  \left(
    1-
    e^{-\int_{t_0}^t a(\tau)\,d\tau}
  \right)M(t)\\
  &\leq M(t)
\end{align*}
componentwise. Hence, by
$G(t)=B(t)-M(t)-Z(t)$ and \eqref{eq_ztpara},
\begin{align}\label{eq:GBequ}
  G(t)
  \leq
  B(t)
  -
  e^{-\int_{t_0}^t a(\tau)\,d\tau}
  \int_{t_0}^t
    e^{\int_{t_0}^s a(\tau)\,d\tau}
    a(s)B(s)\,ds.
\end{align}
 Therefore, using $\sup_{t\geq t_0}\norm{B(t)}\le C$,
\begin{align*}
  \|\pos{G(t)}\|
  &\leq
  \left\|
  \pos{
    B(t)
    -
    e^{-\int_{t_0}^t a(\tau)\,d\tau}
    \int_{t_0}^t
      e^{\int_{t_0}^s a(\tau)\,d\tau}
      a(s)B(s)\,ds
  }
  \right\|\\
  &\leq
  \left\|
    B(t)
    -
    e^{-\int_{t_0}^t a(\tau)\,d\tau}
    \int_{t_0}^t
      e^{\int_{t_0}^s a(\tau)\,d\tau}
      a(s)B(s)\,ds
  \right\|\\
  &\leq
  \|B(t)\|
  +
  e^{-\int_{t_0}^t a(\tau)\,d\tau}
  \int_{t_0}^t
    e^{\int_{t_0}^s a(\tau)\,d\tau}
    a(s)\|B(s)\|\,ds\\
  &\leq
  C+
  C\left(
    1-e^{-\int_{t_0}^t a(\tau)\,d\tau}
  \right)
  \leq2C.
\end{align*}
Thus
$
  \sup_{t\geq t_0}\|\pos{G(t)}\|\leq2C.
$
This proves~(i).

We now prove~(ii). For $a(t)=c/t$, $e^{\int_{t_0}^t a(\tau)\,d\tau}=\frac{t^c}{t_0^c}$. Then, it follows from \eqref{eq:GBequ} that 
\begin{align}\label{eq_Gteq}
  G(t)
  \le 
  B(t)-ct^{-c}\int_{t_0}^t s^{c-1}B(s)\,ds.
 \end{align}
Since $c>0$ and $B(t)\to B_\infty$, L'Hôpital's rule,
applied componentwise, gives
\[
  \lim_{t\to+\infty}
  \frac{c\int_{t_0}^t s^{c-1}B(s)\,ds}{t^c}
  =
  \lim_{t\to+\infty}
  \frac{ct^{c-1}B(t)}{ct^{c-1}}
  =B_\infty.
\]
Then \eqref{eq_Gteq} implies that
\[
  0\leq\norm{\pos{G(t)}}
  \leq
  \norm{B(t)-ct^{-c}\int_{t_0}^t s^{c-1}B(s)\,ds}
  \to0
  \qquad\text{as }t\to+\infty.
\]
and hence
$\lim_{t\to+\infty}\norm{\pos{G(t)}}=0$.

Finally, suppose that, for some component $i$, there exists
$T_i\geq t_0$ such that $M_i(t)=M_i(T_i)$ for all
$t\geq T_i$. By \eqref{eq_ztpara} and
$G(t)=B(t)-M(t)-Z(t)$, we have
\begin{equation}\label{eq_Git}
  G_i(t)
  =
  B_i(t)-ct^{-c}\int_{t_0}^t s^{c-1}B_i(s)\,ds
  -M_i(t)+ct^{-c}\int_{t_0}^t s^{c-1}M_i(s)\,ds.	
\end{equation}
The difference between $B_i(t)$ and its weighted average tends to zero, by the preceding argument.  Since $M_i(t)=M_i(T_i)$ for all $t\geq T_i$,
\[
  ct^{-c}\int_{t_0}^t s^{c-1}M_i(s)\,ds
  =
  ct^{-c}\int_{t_0}^{T_i}s^{c-1}M_i(s)\,ds+
  M_i(T_i)
  \left(
    1-\left(\frac{T_i}{t}\right)^c
  \right).
\]
Therefore,
\[
  \lim_{t\to+\infty}
  ct^{-c}\int_{t_0}^t s^{c-1}M_i(s)\,ds
  =M_i(T_i).
\]
Thus the difference between the weighted average of $M_i$ and $M_i(t)$ tends to zero. Consequently, from \eqref{eq_Git},
\[
  \lim_{t\to+\infty}G_i(t)=0.
\]
This completes the proof of~(ii).
\end{proof}

\begin{lemma}\label{lem:A5}
  Let $u\in C^1([t_0,+\infty);\R^{d_1})$ and let
  $r:[t_0,+\infty)\to\R^{d_2}$ be continuous. Suppose that
  \[
    \int_{t_0}^{+\infty}
    t\bigl(\|\dot u(t)\|^2+\|r(t)\|^2\bigr)\,dt<+\infty.  \]
 If $\{t_k\}\subset[t_0,+\infty)$ satisfies
$t_k\to+\infty$ and $\rho>1$, then
  \[
    \sup_{s\in[t_k,\rho t_k]}
    \|u(s)-u(t_k)\|\to0, \quad \text{as\ } t_k\to +\infty.
  \]
  Moreover, there exists $s_k\in[t_k,\rho t_k]$ such that
  \[
  \lim_{s_k\to+\infty}  s_k\|\dot u(s_k)\|+s_k\|r(s_k)\|=0.
  \]
\end{lemma}

\begin{proof}
For every $s\in[t_k,\rho t_k]$, the Cauchy-Schwarz inequality gives
\begin{align*}
  \|u(s)-u(t_k)\|
  &\leq
  \int_{t_k}^{s}\|\dot u(t)\|\,dt\leq
  \left(
    \int_{t_k}^{s}t\|\dot u(t)\|^2\,dt
  \right)^{1/2}
  \left(
    \int_{t_k}^{s}\frac{dt}{t}
  \right)^{1/2}\\
  &\leq
  \sqrt{\log\rho}
  \left(
    \int_{t_k}^{+\infty}t\|\dot u(t)\|^2\,dt
  \right)^{1/2}.
\end{align*}
 Since $
  \int_{t_0}^{+\infty}t\|\dot u(t)\|^2\,dt<+\infty$ and $t_k\to+\infty$,
its tail integral tends to zero. Taking the supremum over
$s\in[t_k,\rho t_k]$ therefore yields
\[
  \sup_{s\in[t_k,\rho t_k]}
  \|u(s)-u(t_k)\|
  \leq
  \sqrt{\log\rho}
  \left(
    \int_{t_k}^{+\infty}t\|\dot u(t)\|^2\,dt
  \right)^{1/2}
  \to0,
\]
as $t_k\to+\infty$.

To prove the second assertion, set
\[
  q(t):=\|\dot u(t)\|^2+\|r(t)\|^2,
  \qquad
  \varepsilon_k:=
  \int_{t_k}^{\rho t_k}tq(t)\,dt.
\]
The function $q(t)$ is continuous and nonnegative. Moreover, by assumption,
\begin{equation}\label{eq_epK}
	  0\leq\varepsilon_k
  \leq
  \int_{t_k}^{+\infty}tq(t)\,dt
  \to0,\quad \text{as } k\to+\infty.
\end{equation}
For each $k$, choose $s_k\in[t_k,\rho t_k]$ at which $q(t)$
attains its minimum on this  interval. Then
\[
  (\rho-1)t_k q(s_k)
  \leq
  \int_{t_k}^{\rho t_k}q(t)\,dt.
\]
Since $t\geq t_k>0$ on the interval,
\[
  \int_{t_k}^{\rho t_k}q(t)\,dt
  \leq
  \frac1{t_k}
  \int_{t_k}^{\rho t_k}tq(t)\,dt
  =
  \frac{\varepsilon_k}{t_k}.
\]
Consequently,
\[
  q(s_k)
  \leq
  \frac{\varepsilon_k}{(\rho-1)t_k^2}.
\]
Using $s_k\leq\rho t_k$ and
$(a+b)^2\leq2(a^2+b^2)$, from \eqref{eq_epK}, we obtain
\[
  0\leq
  \left(
    s_k\|\dot u(s_k)\|+s_k\|r(s_k)\|
  \right)^2
  \leq
  2s_k^2q(s_k)
  \leq
  \frac{2\rho^2}{\rho-1}\varepsilon_k
  \to 0, \quad\text{as } s_k\to+\infty
\]
This proves both assertions.

\end{proof}

\bibliographystyle{plain}
\bibliography{refer}

\end{document}